\documentclass[11pt]{amsart}
\usepackage{amsmath, amssymb, verbatim,color, amscd}

\newtheorem{theorem}{Theorem}[section]
\newtheorem{lemma}[theorem]{Lemma}

\newtheorem{corollary}[theorem]{Corollary}
\newtheorem{proposition}[theorem]{Proposition}

\renewcommand{\leq}{\leqslant}
\renewcommand{\geq}{\geqslant}

\theoremstyle{definition}

\newtheorem{example}[theorem]{Example}

\theoremstyle{definition}

\numberwithin{equation}{section}

\renewcommand{\leq}{\leqslant}
\renewcommand{\geq}{\geqslant}

\numberwithin{equation}{section} \numberwithin{figure}{section}

\title{Collapsing of negative  K\"{a}hler-Einstein metrics}
 \author[Y. Zhang]{Yuguang Zhang}
 \thanks{The  author is supported in part by  grant NSFC-11271015.}
\address{Yau Mathematical Sciences Center,  Tsinghua University,  Beijing 100084, P.R.China.}
\email{yuguangzhang76@yahoo.com}

\begin{document}
\begin{abstract}
   In this paper, we study the collapsing behaviour of negative K\"{a}hler-Einstein metrics along degenerations of canonical polarized manifolds. We prove that for a  toroidal degeneration of canonical polarized manifolds with the total space $\mathbb{Q}$-factorial, the K\"{a}hler-Einstein metrics on fibers  collapse  to a lower  dimensional complete  Riemannian manifold in the pointed Gromov-Hausdorff sense by suitably choosing the base points. Furthermore, the most collapsed  limit  is a
    real affine K\"{a}hler  manifold.
\end{abstract}
\maketitle
\section{Introduction}
Let  $X$ be  a complex  projective $n$-manifold. We call $X$ a  canonical polarized manifold   if  the  canonical bundle $\mathcal{K}_{X}$ of $X$  is  ample,  and call  $X$  a Calabi-Yau manifold if $\mathcal{K}_{X}$ is trivial.  The Calabi conjecture of the  existence of K\"{a}hler-Einstein metrics   was solved  by Aubin and Yau in the case of canonical polarized manifolds (cf. \cite{aubin,Yau1}), and by Yau for  Calabi-Yau manifolds (cf. \cite{Yau1}).
 More precisely, on a canonical polarized manifold $X$,
 there exists a unique K\"{a}hler-Einstein metric $\omega$ with $\omega\in 2\pi  c_{1} (\mathcal{K}_{X})$ and  negative Ricci curvature, i.e. $${\rm Ric}(\omega)=-\omega,$$  by  \cite{aubin,Yau1}.  On a Calabi-Yau manifold, there are Ricci-flat K\"{a}hler-Einstein metrics by \cite{Yau1}.
   The goal of this paper is to  study the collapsing behaviour of families of   negative K\"{a}hler-Einstein metrics along degenerations in algebro-geometric sense.

   A degeneration of projective   $n$-manifolds $\pi: \mathcal{X}\rightarrow \Delta$ is a flat morphism from a normal Gorenstein   variety $\mathcal{X}$ of dimension $n+1$  to a disc $ \Delta\subset \mathbb{C}$ such that   $X_{t}=\pi^{-1}(t)$, $t\in \Delta^{\ast}=\Delta\backslash \{0\}$,  is smooth except  the central fiber   $X_{0}=\pi^{-1}(0)$.  We denote $X_{0}=\bigcup\limits_{i=1}\limits^{l}X_{0,i}$ and $X_{0,I}=\bigcap\limits_{i\in I} X_{0,i}$, where $X_{0,i}$, $i=1, \cdots, l$, is a irreducible component, and $I\subset \{1, \cdots, l\}$.  If the relative canonical bundle  $\mathcal{K}_{\mathcal{X}/ \Delta}=\mathcal{K}_{\mathcal{X}}\otimes \mathcal{K}_{\Delta}^{-1}$ is relatively  ample, then for any smooth fiber $X_{t}$,   the  canonical bundle $\mathcal{K}_{X_{t}}\cong \mathcal{K}_{\mathcal{X}/ \Delta}|_{X_{t}}$  is ample, and thus $X_{t}$ is a canonical polarized manifold. We call such degeneration   $\pi: \mathcal{X}\rightarrow \Delta$ a canonical polarized degeneration.

 In  \cite{SYZ}, Strominger, Yau and Zaslow proposed  a  conjecture, so called SYZ conjecture,  for constructing mirror Calabi-Yau manifolds via dual special lagrangian fibration.  Later, a new version of the  SYZ conjecture was proposed by   Kontsevich,  Soibelman, Gross and Wilson
    (cf. \cite{GW,KS,KS2}) by using   the collapsing of Ricci-flat K\"{a}hler-Einstein metrics.
  Let   $\mathcal{X} \rightarrow \Delta $ be   a degeneration of  Calabi-Yau $n$-manifolds, i.e. the relative canonical bundle $\mathcal{K}_{\mathcal{X} / \Delta} $ is trivial, and $0\in \Delta$ be  a large complex limit  point (cf. \cite{GHJ}).
   The collapsing   version of SYZ conjecture  asserts that there are Ricci-flat K\"{a}hler-Einstein metrics $\omega_{t}$ on $X_{t}$ for $t\in \Delta^{\ast}$ such that
       $(X_{t}, {\rm diam}_{\omega_{t}}^{-2}(X_{t}) \omega_{t})$ converges to a compact metric space $(B,d_{B})$  in the Gromov-Hausdorff sense, when $t\rightarrow 0$.
 Furthermore, the smooth  locus $B_{0}$ of $ B$ is  open dense,  and  is  of real dimension $n$,  and   admits    a real  affine structure. The metric $d_{B}$ is induced by a Monge-Amp\`ere  metric $g_{B}$ on $B_{0}$, i.e.  under  affine coordinates $x_{1},  \cdots, x_{n}$,  there is a potential function $\phi$ such that $$g_{B}= \sum_{ij} \frac{ \partial^{2} \phi}{ \partial x_{i}  \partial x_{j}} dx_{i} dx_{j},    \   \  {\rm and}  \  \   \det \Big(\frac{\partial^{2}\phi}{ \partial x_{i}  \partial x_{j} }\Big ) =1.$$  Clearly  it  is true for Abelian varieties.
This conjecture was verified by Gross and Wilson for fibred  K3 surfaces with only type $I_{1}$ singular fibers  in \cite{GW}, and was studied for higher dimensional HyperK\"ahler manifolds in \cite{GTZ, GTZ2}.  In  \cite{GTZ2},   Gross-Wilson's result was extended  to all elliptically fibred K3 surfaces.

   Inspired by  this  collapsing version of  SYZ conjecture, we study the    limits  of negative K\"{a}hler-Einstein metrics on  canonical polarized   manifolds degenerating to some singular varieties.

 Let $\pi: \mathcal{X}\rightarrow \Delta$ be a canonical polarized degeneration such that $X_{0}$  has only simple normal crossing singularities, i.e. $X_{0}$ is reduced,    locally given by $z_{1}\cdot\cdots \cdot  z_{k}=0$ under   local  coordinates $z_{1}, \cdots, z_{n}$ on  $\mathcal{X}$, and  any  $X_{0,I}$ is smooth.  Let  $\omega_{t}\in 2\pi  c_{1} (\mathcal{K}_{X_{t}})$, $t\in \Delta^{*}$, be the unique K\"{a}hler-Einstein metric on $X_{t}$.
 The convergence of $\omega_{t}$  was studied by various authors (cf. \cite{Tian,leung,ruan1,ruan2,ruan4}).  In   \cite{Tian}, it is proved  that $\omega_{t}$ converges smoothly  to a complete K\"{a}hler-Einstein $\omega_{0}$  with negative Ricci curvature on the regular locus $X_{0,reg}=\bigcup\limits_{i=1}\limits^{l}X_{0,i,reg}$ in the  Cheeger-Gromov sense,   if an additional condition that   any three of the components $X_{0,i}$ have  empty intersection   is satisfied.  More precisely,  for any   smooth family of embeddings $F_{t}:X_{0,reg}  \rightarrow X_{t}$, we have that  $$F_{t}^{*} \omega_{t} \rightarrow \omega_{0},  \   \  \  {\rm when}  \ \  t\rightarrow 0, $$ in the locally  $C^{\infty}$-sense on   $ X_{0,reg}$, where  $\omega_{0}$ is the complete  K\"{a}hler-Einstein metric  on  $ X_{0,reg}$  previously obtained in \cite{TY,Kob,cheng}.    In \cite{ruan1,leung}, the additional assumption is removed, and furthermore, the result is generalized to the case of  toroidal degenerations in \cite{ruan2}. These theorems describe     the  the non-collapsing part   of the limit of $(X_{t}, \omega_{t})$.

 Since the volume of $\omega_{0}$ is finite,  there must be some  collapsing part when $(X_{t}, \omega_{t})$ approaches to the limit,  i.e. there are   points $p_{t}\in X_{t}$ such that the volumes of metric $1$-balls satisfy   $$ {\rm Vol}_{\omega_{t}}(B_{\omega_{t}}(p_{t},1)) \rightarrow 0,  \   \  \  {\rm when}  \ \  t\rightarrow 0.  $$
 Now by  Gromov's precompactness theorem (cf.  \cite{Gromov}),
 a sequence of  $(X_{t}, \omega_{t}, p_{t})$ converges to a pointed complete metric space $(W,d_{W},p_{\infty})$ of Hausdorff dimension less than $2n$ in the pointed Gromov-Hausdorff sense, i.e. for any $R>0$, the  metric $R$-ball   $(B_{\omega_{t}}(p_{t},R), \omega_{t})$ converges to the metric $R$-ball  $ (B_{d_{W}}(p_{\infty},R), d_{W})$  in the Gromov-Hausdorff sense (cf. \cite{Fukaya1}).

 The following theorem is a special case of the main theorem (Theorem \ref{main}) of the present paper,  where a more general hypothesis  is assumed.

 \begin{theorem}\label{thm01}
 Let $\pi: \mathcal{X}\rightarrow \Delta$ be a canonical polarized degeneration such that $X_{0}$  has only simple normal crossing singularities, and $\omega_{t}\in 2\pi  c_{1} (\mathcal{K}_{X_{t}})$ be the unique K\"{a}hler-Einstein metric on $X_{t}$, $t\in \Delta^{*}$.  For any $X_{0,I}$
   and any point $p_{0}\in X_{0,I}\backslash \bigcup\limits_{i\notin I}X_{0,i}$, there are points   $p_{t}\in X_{t}$ such that   $p_{t}\rightarrow p_{0}$ in $\mathcal{X}$ when $t\rightarrow 0$,and  by passing to a sequence,     $(X_{t}, \omega_{t}, p_{t})$ converges to a complete Riemannian manifold $(W, g_{W}, p_{\infty})$ with   $\dim_{\mathbb{R}}W= 2 n+1-\sharp I$ in the pointed Gromov-Hausdorff sense. Furthermore,
  if   $\dim_{\mathbb{C}}X_{0,I}=0$, then $(W, g_{W})$ is isometric to
$(B,g_{B})$ by suitably  choosing $p_{t}$, where $B$ is the interior of the  standard simplex in $\mathbb{R}^{n}$, and there is a smooth  potential function $\phi$ on $B$ such that $ \phi|_{\partial \overline{B}}=+\infty$,  $$g_{B}=\sum_{ij=1}^{n}\frac{\partial^{2} \phi}{\partial x_{i}\partial x_{j}}d  x_{i}d x_{j},
\  \   and  \  \  \det \Big(\frac{\partial^{2} \phi}{\partial x_{i}\partial x_{j}}\Big)=\kappa e^{2\phi}, $$ for a constant $\kappa >0$.\end{theorem}

Actually $(X_{t}, \omega_{t})$  collapses smoothly in a certain sense, which  is stronger than  the Gromov-Hausdorff topology  (See Theorem \ref{main}  for details).

 This theorem  shows  a similar collapsing behaviour to the SYZ conjecture for Calabi-Yau manifolds, i.e. under certain assumptions, the limit metric space $W$ is an affine K\"{a}hler manifold of  real dimension $n$, and the potential function satisfies a real  Monge-Amp\`{e}re equation.
However,  unlike  the Calabi-Yau case,  we always have  the non-collapsing  part of the limit,  and we do not rescale the metric to obtain the collapsing  limit.
 Note that for algebraic curves of higher  genus, the rescaled limit exists, and is a compact  metric graph by \cite{Ok}.  However, we do not expect that  still holds in the  higher dimensional case.

In the original SYZ conjecture (cf. \cite{SYZ}), the existence of special lagrangian submanifolds is expected when Calabi-Yau manifolds are near    the large complex limit. As an application, we will  construct some generalized special lagrangian submanifolds on canonical polarized manifolds (See Section 2.3 for details).

The understanding of the limit behaviour of negative K\"{a}hler-Einstein metrics is also required for other program. The moduli space $\mathcal{M}$ of canonical  polarized manifolds with a fixed Hilbert polynomial was proven to be a quasi-projective manifold by  Viehweg   in \cite{Vie}, and the recent progress on the moduli space of stable varieties (cf. \cite{Kol}) gives a natural algebro geometric  compactification $\overline{\mathcal{M}}$ of $\mathcal{M}$.   Meanwhile, the existence of  singular K\"{a}hler-Einstein metrics on  stable varieties  was obtained in \cite{BG}.  A natural question is to understand such compactification from the differential geometric viewpoint (cf. \cite{BG,Song}), for example in the Gromov-Hausdorff sense or the Weil-Petersson geometry  sense. However unlike  the case of  Calabi-Yau manifolds  (cf. \cite{Zhang,Tosatti}), we   would not have the coincidence of the  Gromov-Hausdorff non-collapsing convergence and the finite Weil-Petersson distance. In  Theorem \ref{thm01},  $(X_{t}, \omega_{t})$ diverges in the  Gromov-Hausdorff sense, but the   Weil-Petersson metric on  $\Delta^{*}$ is not complete, i.e. $\{0\}$ has finite Weil-Petersson distance to the interior by \cite{Tian,ruan1,ruan2}.

 This paper is organized  as the followings.  In Section 2,  we introduce the preliminary materiel  and state the main theorems  (Theorem \ref{main} and Theorem \ref{main2}) of this paper. In Section 1.1, we   construct some semi-flat  K\"{a}hler-Einstein metrics from those affine K\"{a}hler metrics obtained by Cheng and Yau previously. In Section 1.2 and Section 1.3, the main theorems (Theorem  \ref{main} and Theorem \ref{main2}) are  given.  Theorem  \ref{main} study  the metric collapsing along toroidal degenerations, and Theorem \ref{main2} shows  the existence of  generalized special lagrangian submanifolds.   Section 3 is devoted to prove Theorem \ref{main}. Firstly,   we construct the approximation background metrics in Section 3.1,  then we do some local calculations and  prove Theorem  \ref{main} in Section 3.2.  The last section proves Theorem \ref{main2}.

\noindent {\bf Acknowledgements:} The author would  like to thank Prof.  	
Shiu-Yuen Cheng for answering a question.

\section{Main Theorems}
In this paper, we always denote  $N\cong \mathbb{Z}^{n+1}$,    $N_{\mathbb{R}}=N\otimes_{\mathbb{Z}}\mathbb{R}$, $T_{N}= N\otimes_{\mathbb{Z}}\mathbb{C}^{*}$,  $M=\hom_{\mathbb{Z}}(N,\mathbb{Z})$ and $M_{\mathbb{R}}=M\otimes_{\mathbb{Z}}\mathbb{R}$.

\subsection{Semi-flat K\"{a}hler-Einstein  metric}
 In this section, we recall  a theorem due to Cheng and Yau for the existence of  affine K\"{a}hler metrics, which induce some  semi-flat K\"{a}hler-Einstein metrics that appear in the main theorem.

Let $\sigma$ be a rational strongly convex polyhedral cone in $ N_{\mathbb{R}}$, and $\check{\sigma}\subset M_{\mathbb{R}}$ be the dual cone. If $u_{\sigma}\in M\cap \check{\sigma}$  satisfies  $\langle u_{\sigma}, v\rangle=1$ for the  primitive lattice   vector $v\in \tau\cap N$ of   any  1-dimensional face $\tau$ of $\sigma$, then we define $$\Lambda_{\mathbb{R}} =\{ v \in N_{\mathbb{R}}| \langle v, u_{\sigma} \rangle=1\},   \   \ B_{\sigma}= \Lambda_{\mathbb{R}} \cap {\rm Int} (\sigma), \  \ {\rm and} \ \ \Lambda=N\cap \Lambda_{\mathbb{R}}$$ where ${\rm Int} (\sigma)$ denotes the interior of $\sigma$. The closure $\overline{B}_{\sigma}$ of $B_{\sigma}$ is a  rational  convex  polytope in $\Lambda_{\mathbb{R}}$.

Let  $\mathcal{Y}_{\sigma}$ be  the affine  toric variety associated to $\sigma$, i.e. $\mathcal{Y}_{\sigma}={\rm Spec}(\mathbb{C}[\check{\sigma}\cap M])$, and $t=\mathcal{Z}^{u_{\sigma}}: \mathcal{Y}_{\sigma} \rightarrow \mathbb{C}$.  We have a family of varieties $Y_{\sigma,t}={\rm div}(\mathcal{Z}^{u_{\sigma}}-t)$ degenerating  to the toric boundary  $Y_{0}$, i.e.   $Y_{0}=\bigcup\limits_{i=1}^{d}D_{i}$ where $D_{i}$ is a primitive toric Weil divisor.

If  $e_{0}, \cdots, e_{n}\in N$ is a   basis,  we denote  $x_{0}, \cdots , x_{n}$  the respecting coordinates on $N_{\mathbb{R}}$, and denote  $z_{j}=\mathcal{Z}^{e_{j}^{*}}$, $j=0, \cdots, n$.  If $u_{\sigma}=\sum\limits_{j=0}^{n}m_{j}e_{j}^{*}$, then $Y_{\sigma,t}$ is given by $z_{0}^{m_{0}}\cdot\cdots\cdot z_{n}^{m_{n}}=t$,  and $\Lambda_{\mathbb{R}}$ is given by $m_{0}x_{0}+\cdots + m_{n}x_{n}=1$.
   Without loss of generality, we assume that $x_{1}, \cdots , x_{n}$ are  coordinates on $\Lambda_{\mathbb{R}}$, i.e. $m_{0}\neq 0 $,  which give  an integral affine structure on $B_{\sigma}$.

    For any $t\in \Delta^{*}$,  the logarithmic map is  $$ {\rm Log}_{t}:  T_{N} \rightarrow N_{\mathbb{R}} ,  \  \  {\rm by  } \ \ z_{j} \mapsto  x_{j}=\frac{\log |z_{j}|}{\log |t|}, \  j=0, \cdots, n.$$
   It is clear that
    $ {\rm Log}_{t} (  Y_{\sigma,t} )= \Lambda_{\mathbb{R}}$.
      We denote $$\mathcal{U}=\{p\in \mathcal{Y}_{\sigma}| |\mathcal{Z}^{u_{k}}(p)|<1, k=1, \cdots, d' \},$$ which is  an open subset of $\mathcal{Y}_{\sigma}$,  where $u_{k}\in M \cap\check{\sigma}$ such that $\sigma=\{v\in N_{\mathbb{R}} | \langle v, u_{k}\rangle \geq0, k=1, \cdots, d'\}$. We have  ${\rm Log}_{t} (  \mathcal{U} )= {\rm Int} (\sigma) $, and moreover, $ {\rm Log}_{t} (  Y_{\sigma,t}\cap \mathcal{U} )= B_{\sigma}$.

     We define coordinates  $\theta_{1}, \cdots , \theta_{n}$  on $\Lambda_{\mathbb{R}}$  by $\theta_{j}=dx_{j}$, $j=1, \cdots, n$, under the identification of  the tangent bundle $T   B_{\sigma} \cong B_{\sigma} \times \Lambda_{\mathbb{R}}$.  Then there is a natural complex structure on $B_{\sigma} \times \sqrt{-1} \Lambda_{\mathbb{R}}$ given by complex coordinates $w_{j}=x_{j}+\sqrt{-1}\theta_{j}$, $j=1, \cdots, n$, which  induces a complex structure on $Y_{t,m_{0}}(B_{\sigma})= B_{\sigma} \times \sqrt{-1} (\Lambda_{\mathbb{R}}/ \frac{2\pi m_{0} \Lambda}{\log |t|})$ for any $t\in \Delta^{*}$. We define a finite covering map $q_{\sigma}: Y_{t,m_{0}}(B_{\sigma}) \rightarrow Y_{\sigma,t}\cap \mathcal{U} $  by setting   $z_{j}=\exp ((\log |t|)w_{j})$, $j=1, \cdots,n$, and $$z_{0}=\exp (\frac{1}{m_{0}}\log |t|+\sqrt{-1}\frac{\arg (t)}{m_{0}}-\sum_{j=1}^{n}\frac{m_{j}}{m_{0}}(\log|t|)w_{j}).  $$ Furthermore,  $f_{t}={\rm Log}_{t}|_{Y_{\sigma,t}\cap \mathcal{U}}: Y_{\sigma,t}\cap \mathcal{U} \rightarrow B_{\sigma}$ is a fibration such   that $f_{t}\circ q_{\sigma}$ is  the projection from $Y_{t,m_{0}}(B_{\sigma})$ to $B_{\sigma}$.

 Now we recall a  theorem for    the existence of affine  K\"{a}hler  metrics   in \cite{CY1}.

\begin{theorem}[Theorem 4.4 in \cite{CY1}]\label{CY}  For any constant $\kappa >0$, there is a smooth convex  solution $\phi$ of the real  Monge-Amp\`{e}re equation  \begin{equation}\label{e0.1+}\det \Big(\frac{\partial^{2} \phi}{\partial x_{i}\partial x_{j}}\Big)=\kappa e^{2\phi},  \  \  \ \phi|_{\partial \overline{B}_{\sigma}}=+\infty,\end{equation} and $$g_{B_{\sigma}}=\sum_{ij=1}^{n}\frac{\partial^{2} \phi}{\partial x_{i}\partial x_{j}}d  x_{i}d x_{j}$$ is a  complete affine   K\"{a}hler metric on $B_{\sigma}$.
\end{theorem}

  Note that the constant $\kappa$ is chosen to be $1$ in   \cite{CY1}, and however, we can obtain  the general case by rescaling the coordinates.
By pulling back $\phi$, we  regard $\phi$ as a function on $B_{\sigma} \times \sqrt{-1} \Lambda_{\mathbb{R}}$, i.e.  $\phi(w_{1}, \cdots, w_{n})=\phi(x_{1}, \cdots, x_{n})$, which defines  a complete K\"{a}hler metric \begin{equation}\label{e0.1} \omega^{sf}=2\sqrt{-1} \partial \overline{\partial} \phi= \frac{\sqrt{-1}}{2}\sum_{ij=1}^{n}\frac{\partial^{2} \phi}{\partial x_{i}\partial x_{j}}d  w_{i}\wedge d \bar{w}_{j}
  \end{equation} on $B_{\sigma} \times \sqrt{-1} \Lambda_{\mathbb{R}}$.    By (\ref{e0.1+}), $\phi$ satisfies the complex Monge-Amp\`{e}re equation $\det \big(\frac{\partial^{2} \phi}{\partial w_{i}\partial \bar{w}_{j}}\big)=4^{n}\kappa e^{2\phi} $ on $B_{\sigma} \times \sqrt{-1} \Lambda_{\mathbb{R}}$, and hence $\omega^{sf}$ is a K\"{a}hler-Einstein  metric with Ricci curvature $-1$, i.e. $$ {\rm Ric}(\omega^{sf})=-\sqrt{-1} \partial \overline{\partial}\log \det \big(\frac{\partial^{2} \phi}{\partial w_{i}\partial \bar{w}_{j}}\big)=- \omega^{sf}. $$    Now Proposition 5.5 in \cite{CY} implies that $\phi$ is the unique solution of (\ref{e0.1+}) (See also  \cite{Hild}).

   Since both $\phi$ and $\omega^{sf}$ are   invariant under the translation $w_{j}\mapsto w_{j}+\sqrt{-1}  \lambda$ for any $\lambda\in \mathbb{R}^{1}$, $\omega^{sf}$ descents to    a complete K\"{a}hler-Einstein  metric  on $Y_{t,m_{0}}(B_{\sigma})$ first, for any $t\in \Delta^{*}$, and further to a  complete K\"{a}hler-Einstein  metric  on   $ Y_{\sigma,t}\cap \mathcal{U} $ denoted  by   $\omega^{sf}_{t}$. Note that the corresponding Riemannian metric of  $\omega^{sf}$ is $$g^{sf}= \sum_{ij=1}^{n}\frac{\partial^{2} \phi}{\partial x_{i}\partial x_{j}}(d  x_{i}d x_{j}+d \theta_{i}d \theta_{j}).  $$
   The first consequence  is that the restriction of  $\omega^{sf}_{t}$ on any fiber $f^{-1}_{t}(x) $, $ x\in B_{\sigma}$,  is  flat, so called a semi-flat K\"{a}hler-Einstein  metric.  The second one is that the diameter of the fiber $${\rm diam}_{\omega^{sf}_{t}}(f^{-1}_{t}(x))\sim - ( \log|t|)^{-1} \rightarrow 0, $$ and by suitably  choosing a family of base points $p_{t}\in Y_{\sigma,t}$, $(Y_{\sigma,t}\cap \mathcal{U}, \omega^{sf}_{t},p_{t}) $ converges to $(B_{\sigma}, g_{B_{\sigma}}, p_{\infty})$ in the pointed Gromov-Hausdorff sense, when $t\rightarrow 0$. We say that $(Y_{\sigma,t}\cap \mathcal{U}, \omega^{sf}_{t}) $ collapses  to $(B_{\sigma}, g_{B_{\sigma}})$.

   In summary, we have the following proposition.

\begin{proposition}\label{pro-0}   For any $t\in\Delta^{*}$, there is a unique complete   K\"{a}hler-Einstein  metric $\omega^{sf}_{t}$ on $Y_{\sigma,t}\cap \mathcal{U}$ such that the Ricci curvature is $-1$, i.e.
   $$ {\rm Ric}(\omega^{sf}_{t})=-\omega^{sf}_{t},$$ and $\omega^{sf}_{t}$ is semi-flat respecting to the torus fibration $f_{t}: Y_{\sigma,t}\cap \mathcal{U} \rightarrow B_{\sigma}$.   Furthermore,   $(Y_{\sigma,t}\cap \mathcal{U}, \omega^{sf}_{t},p_{t}) $ converges to $(B_{\sigma}, g_{B_{\sigma}},p_{\infty})$ in the pointed Gromov-Hausdorff sense  by choosing a family of base points $p_{t}\in Y_{\sigma,t}$,  when $t\rightarrow 0$.
\end{proposition}

The logarithm ${\rm Log}_{t}$ is used to convert  classical algebraic varieties to tropical varieties (cf. \cite{Mikh}), and it is believed that the collapsing of K\"{a}hler-Einstein  metrics can do the same in  certain circumstances  (cf. \cite{gross-survey,Fukaya2}).  This is true in our case as   a direct corollary of the previous  arguments.

  Let $ \mathfrak{p}\in \mathbb{C}[\check{\sigma}\cap M](t)$,   i.e. $ \mathfrak{p}=\sum\limits_{u\in A} b_{u}t^{\upsilon (u)}\mathcal{Z}^{u}$ for a finite set $A\subset \check{\sigma}\cap M$, $b_{u} \in \mathbb{C}^{*}$,  and $\upsilon: A\rightarrow \mathbb{Z}$, and  $ \mathrm{V}_{t,\mathfrak{p}} \subset Y_{\sigma,t}$ be  the variety defined  by $\mathfrak{p}|_{Y_{\sigma,t}}=0$.    The image $\mathcal{A}_{t}={\rm Log}_{t} ( \mathrm{V}_{t,\mathfrak{p}}) \subset \Lambda_{\mathbb{R}}$ is called  an amoeba, and it is proven  in  \cite{Mikh} that $\mathcal{A}_{t}$ converges to a polyhedron complex $\mathcal{A}_{\infty}$ in the Hausdorff topology,  when $t\in \mathbb{R}$ and $t\rightarrow 0$.  Here  $\mathcal{A}_{\infty}$ is called a non-Archimedean amoeba, and is the set of  non-smooth points of the function
$$ \mathfrak{p}_{\infty}(x) = \min\limits_{u\in A} \{\upsilon (u)+\langle x, u \rangle\}$$ on $\Lambda_{\mathbb{R}}$.    In tropical geometry, $\mathcal{A}_{\infty}$ is the  tropical hypersurface defined by $\mathfrak{p}$ (cf. \cite{Mikh}).  We have the following  corollary by the collapsing of  $\omega^{sf}_{t}  $  to $ g_{B_{\sigma}}$.

\begin{corollary}\label{coro}    When $t\in \mathbb{R}$ and $t\rightarrow 0$,  $$\mathrm{V}_{t,\mathfrak{p}}\cap \mathcal{U} \rightarrow \mathcal{A}_{\infty}\cap B_{\sigma}$$ under  the pointed Gromov-Hausdorff convergence of $(Y_{\sigma,t}\cap \mathcal{U}, \omega^{sf}_{t}) $  to $(B_{\sigma}, g_{B_{\sigma}})$.
\end{corollary}

\subsection{Toroidal degeneration}
A  degeneration   $\pi: \mathcal{X}\rightarrow \Delta$ is called simple  toroidal, if for any point $x\in \mathcal{X}$,  there is
  an open neighborhood  $U$ satisfying that
  \begin{itemize}
  \item[i)] $U$ is  isomorphic to an open subset of an affine toric variety $\mathcal{Y}_{\sigma}$, denoted still by $U$.
   \item[ii)]  The restriction of $\pi$ on $U$ is given by a regular function $\mathcal{Z}^{u_{\sigma}}$, where  $u_{\sigma}\in M\cap \check{\sigma}$  satisfies  $\langle u_{\sigma}, v\rangle=1$ for the primitive lattice   vector $v\in \tau\cap N$ of   any  1-dimensional face $\tau$ of $\sigma$. Hence if $D_{1}, \cdots, D_{d}$ are primitive   toric Weil divisors of $\mathcal{Y}_{\sigma}$, then we have that     $X_{0}\cap U= \sum\limits_{j=1}^{d} D_{j} \cap U$, and  $X_{0}$ is reduced.
    \item[iii)] Any non-empty   $X_{0,I}$  is connected and  normal, which implies that any $X_{0,I}$ does not intersect with itself.
  \end{itemize}
       Since
 the canonical divisor $\mathcal{K}_{\mathcal{Y}_{\sigma}}=-\sum\limits_{j=1}^{d} D_{j}$ (cf. \cite{Fulton}), we have  that
  $\mathcal{K}_{\mathcal{X}}|_{U}=- {\rm div}(\mathcal{Z}^{u_{\sigma}})$, and thus $\mathcal{K}_{\mathcal{X}}$  is Cartier, i.e.  $\mathcal{X}$ is Gorenstein. Degenerations with only simple normal crossing singularities are special cases of simple  toroidal degenerations.

In Chapter II of \cite{Mum}, a compact polyhedral complex $\mathcal{B}$ with integral structure, called the dual intersection complex,   is associated to $\pi: \mathcal{X}\rightarrow \Delta$ such that  cells of $\mathcal{B}$ are in one-to-one correspondence to those  non-empty $X_{0,I}$. More precisely, for any $X_{0,I}\neq \emptyset$, there is a unique polyhedral  cell $\overline{B}_{I}\in \mathcal{B}$ such that  $\dim_{\mathbb{R}} \overline{B}_{I}=n-\dim_{\mathbb{C}}X_{0,I}$, and $\overline{B}_{I'}$ is a face of $\overline{B}_{I}$ if and only if  $X_{0,I'}\supset X_{0,I}$.  The cell $\overline{B}_{I}\in \mathcal{B}$ associated to $X_{0,I}$ is constructed as the following. Let   $p\in X_{0,I}\backslash \bigcup\limits_{j \notin I} X_{0,j}$, and $U\subset \mathcal{X}\backslash \bigcup\limits_{j \notin I} X_{0,j}$ be a neighborhood of $p$ isomorphic  to an open subset of an affine  toric variety $\mathcal{Y}_{\sigma}$. If  $\sigma$ is the corresponding rational  convex   cone in $N_{\mathbb{R}}$, then $$\overline{B}_{I}=\{v\in \sigma| \langle v, u_{\sigma}\rangle =1 \}. $$ We denote $B_{I}$   the interior of $\overline{B}_{I}$.

Now we state the main theorem of the present paper.

\begin{theorem}\label{main}  Let $\pi: \mathcal{X}\rightarrow \Delta$ be a simple  toroidal canonical  polarized  degeneration of projective   $n$-manifolds,
 and $\omega_{t}$  be the unique  K\"{a}hler-Einstein metric in  $2 \pi c_{1}(\mathcal{K}_{X_{t}})$, $t\in \Delta^{*}$. If  $\mathcal{X}$ is $\mathbb{Q}$-factorial, then
the  followings hold.
    \begin{itemize}
  \item[i)]  For any $X_{0,I}$ with $\sharp I>1$,
   and any point $p_{0}\in X_{0,I}\backslash \bigcup\limits_{i\notin I}X_{0,i}$, there are points  $p_{t}\in X_{t}$  such that  $p_{t}\rightarrow p_{0}$ in $\mathcal{X}$ when $t\rightarrow 0$, and by passing to  a  sequence,     $(X_{t}, \omega_{t}, p_{t})$ converges to a complete Riemannian manifold $(W, g_{W}, p_{\infty})$ with   $\dim_{\mathbb{R}}W= \dim_{\mathbb{R}}\overline{B}_{I}+ 2 \dim_{\mathbb{C}}X_{0,I} $ in the pointed Gromov-Hausdorff sense.
  \item[ii)]  If  $\dim_{\mathbb{C}} X_{0,I}=0$, then $(W, g_{W})$ is isometric  to $(B_{I},g_{B_{I}})$ by suitably  choosing  $p_{t}$, where  $g_{B_{I}}$ is the complete affine K\"{a}hler metric obtained  in Theorem \ref{CY}. Furthermore, if $\omega^{sf}_{t,I}$ is the semi-flat K\"{a}hler-Einstein metric constructed from $g_{B_{I}}$ in Proposition \ref{pro-0} on a neighborhood of  $U\cap X_{t}$, where $U$ is a neighborhood of $X_{0,I}$ isomorphic an open subset of a toric variety,   then
         $$\|\omega_{t}-\omega^{sf}_{t,I}\|_{C^{\nu}_{loc}(X_{t}\cap U,  \omega^{sf}_{t,I})}\rightarrow 0,$$ for any $\nu>0$, when $t\rightarrow 0$, i.e. the collapsing is in the $C^{\infty}$-sense, and the convergence do not need to pass any sequence.
  \end{itemize}
 \end{theorem}

This theorem describes the  collapsed  limits  of $\omega_{t}$,    while the previous  results of  \cite{Tian,leung,ruan1,ruan2} describe  the non-collapsed limits, i.e. they still have complex dimension $n$.

The notion of    toroidal   degeneration is an   algebro-geometric analogue  of $F$-structure introduced in \cite{CC1}.  An $F$-structure $\mathcal{F}$ on a smooth manifold $X$ consists an open covering $\{U_{\alpha}\}$ such that for each $U_{\alpha}$, there is an effective $T^{n_{\alpha}}$-action on a finite cover of $U_{\alpha}$, and on any overlap $U_{\alpha}\cap U_{\beta}$, these two torus actions $T^{n_{\alpha}}$ and $T^{n_{\beta}}$ are compatible in a certain sense (See \cite{Fukaya2} for the details).  For  a toroidal   degeneration $\pi: \mathcal{X}\rightarrow \Delta$, a small neighborhood $U$ of a $X_{0,I}$ with $\sharp I>1$ is isomorphic to an open subset of a toric variety, and $X_{t}\cap U$ is given by a monomial. Thus there is a natural local $T^{n_{\alpha}}$-action on $X_{t}\cap U$. We conjecture that there is an $F$-structure $\mathcal{F}$ on $X_{t} \cap \mathfrak{U}$, where $\mathfrak{U}$ is   a small neighborhood of $\bigcup\limits_{\sharp I >1}X_{0,I}$ in $\mathcal{X}$, and more importantly,  this $\mathcal{F}$ is Hamiltonian, i.e. there is a symplectic form $\varpi_{t}$ on $X_{t}$ such that   any local torus action of $\mathcal{F}$  is Hamiltonian.

Theorem \ref{main} and Proposition \ref{prop22} in Section 3.2 show that the K\"{a}hler-Einstein metric $\omega_{t}$ approximates some local semi-flat K\"{a}hler-Einstein metrics  $\omega^{sf}_{t,I}$ on small open subsets of $X_{t}$, and $\omega^{sf}_{t,I}$ collapses smoothly to lower dimensional spaces along  local torus fibrations.
 Moreover, we would see that the curvature of $\omega_{t}$ is bounded independent of $t$ in Section 3.1. Hence there is an   $F$-structure $\mathcal{F}'$ on some region of $X_{t}$ by \cite{CC2}, and we again conjecture that   $\mathcal{F}'$ can be made to coincide with the above $\mathcal{F}$.   Hamiltonian $F$-structures would be studied in a separate paper.

 We remark that Theorem \ref{main} should hold for more general settings, for example,  toroidal   degenerations without the assumption of $\mathcal{X}$ being $\mathbb{Q}$-factorial as in \cite{ruan3}, or the log  pair case, i.e. $\mathcal{K}_{\mathcal{X}/\Delta}+D$ is ample for a Cartier  divisor $D$, as in \cite{Tian,leung}.  For avoiding too many technique difficulties, we leave those generalizations for future studies.   In a  recent paper \cite{BG},  the existence of singular K\"{a}hler-Einstein metrics is obtained  for stable varieties, i.e. varieties with semi-log canonical singularities and ample canonical divisor. It is also expected   that the convergence theorems  of \cite{Tian,leung,ruan1,ruan2,ruan4} can be generalized to degenerations with central fiber $X_{0}$ stable varieties (cf. \cite{BG,Song}), which is  related to the question of differential geometric understanding of  the moduli space for stable varieties.

 We finish this section by showing  an example that Theorem \ref{main} and Theorem  \ref{thm01}  can apply.

 \begin{example} Firstly, we recall the standard Mumford degeneration of toric varieties.  Let $M'\cong \mathbb{Z}^{n}$ such that $M\cong M'\times \mathbb{Z}$,  and  $\mathcal{P}\subset M'_{\mathbb{R}}=M'\otimes_{\mathbb{Z}}\mathbb{R}$ be a lattice polytope. If $\psi: \mathcal{P}\rightarrow \mathbb{R}$ is a piecewise linear convex  function respecting to a lattice  polyhedral decomposition $ \mathfrak{P}$ of $\mathcal{P}$ with integral slopes, we define a lattice  polyhedron $$\tilde{\mathcal{P}}=\{(v,r)\in M_{\mathbb{R}}\cong M'_{\mathbb{R}}\times \mathbb{R}|  \psi (v)\leq r \},$$ which determines a toric variety $X_{\mathcal{P}}$ with a regular function $\pi=\mathcal{Z}^{(0,1)}: X_{\tilde{\mathcal{P}}} \rightarrow \mathbb{C}$.   For any $t\in \mathbb{C}\backslash \{0\}$, $X_{t}=\pi^{-1}(t)$ is isomorphic to  the toric variety $X_{\mathcal{P}}$  associated  to $\mathcal{P}$, and $X_{0}=\pi^{-1}(0)=\bigcup\limits_{\tau \in \mathfrak{P}_{{\rm max}}}X_{\tau}$, where $\mathfrak{P}_{{\rm max}}$ denotes the set of $n$-dimensional polytopes of $ \mathfrak{P}$, and $X_{\tau}$ is the toric variety associated  to $\tau\in \mathfrak{P}_{{\rm max}}$. By choosing $\mathcal{P}$ and $\psi$ properly, we can assume that  $X_{0}$ has only simple normal crossing singularities, and $X_{t}$ is smooth for any  $t\neq 0$.  For instance, we take $\mathcal{P}$, $ \mathfrak{P}$   and $\psi$ as the following:

 \setlength{\unitlength}{0.4cm}
\begin{picture}(6,4)
\thicklines
\put(1,1){\line(0,1){2}}\put(1,3){\line(1,0){2}}\put(3,1){\line(0,1){2}}
\put(1,1){\line(1,0){2}}
 \put(1,2){\line(1,0){1}}
  \put(2,1){\line(0,1){1}}
  \put(2,2){\line(1,1){1}}
  \put(1,1){\circle*{0.3}}
   \put(1,2){\circle*{0.3}}
    \put(2,1){\circle*{0.3}}
     \put(2,2){\circle*{0.3}}
       \put(1,3){\circle*{0.3}}
               \put(3,1){\circle*{0.3}}    \put(2,3){\circle*{0.3}}     \put(3,2){\circle*{0.3}}     \put(3,3){\circle*{0.3}}
          \put(-1,2){$-u_{1}$}\put(1,0){$-u_{2}$} \put(3,3){$u_{1}+u_{2}$}
          \put(5,2){$\psi(-u_{1})=0$, $\psi(-u_{2})=0$, $\psi(u_{1}+u_{2})=1$.}
\end{picture}

 Now we follow the argument in the proof of Lemma 1.4 in \cite{Kol2}. Let $H$ be a sufficiently general  very ample divisor on $X_{\tilde{\mathcal{P}}}$ such that $\mathcal{K}_{ X_{\tilde{\mathcal{P}}}}\otimes \mathcal{O}(H) $ is   ample, and   $H+X_{t}$  has     simple normal crossing singularities for any $|t|<\varepsilon \ll 1$. If  $\mathfrak{c}:  \tilde{X}_{\tilde{\mathcal{P}}} \rightarrow X_{\tilde{\mathcal{P}}}$ is the   double ramified cover along $2H$, then the Hurwitz formula shows that $\mathcal{K}_{ \tilde{X}_{\tilde{\mathcal{P}}}}\cong \mathfrak{c}^{*}(\mathcal{K}_{ X_{\tilde{\mathcal{P}}}}\otimes \mathcal{O}(H))$, and hence, $\mathcal{K}_{ \tilde{X}_{\tilde{\mathcal{P}}}}$ is ample.  Note that $\tilde{X}_{0}=\mathfrak{c}^{-1}(X_{0})$ still has only simple normal crossing singularities, and for any $t$ with $0<|t|\ll 1$, $\tilde{X}_{t}=\mathfrak{c}^{-1}(X_{t})$ is  smooth.  We obtain  a canonical degeneration $\tilde{\pi}: \mathcal{X}\rightarrow \Delta\subset \mathbb{C}$ satisfying the hypothesises  in Theorem   \ref{main} and Theorem  \ref{thm01} by letting $\tilde{\pi}=\pi\circ \mathfrak{c}$ and $\mathcal{X}=\tilde{\pi}^{-1}(\Delta)$.

 \end{example}

\subsection{Special lagrangian submanifold}
The original SYZ conjecture  asserts  the existence of special lagrangian submanifolds  when Calabi-Yau manifolds are near    the large complex limit (cf. \cite{SYZ}). There are some attempts to generalize the  SYZ conjecture to the case of canonical polarized manifolds (cf. \cite{LW}), which include analog notions for special lagrangian submanifold.  We also like to study   a generalization of  special lagrangian submanifold.

 If  $X$ is  a canonical polarized  projective $n$-manifold, then by definition, the canonical bundle $\mathcal{K}_{X}$ is ample.
    Let $\Omega$ be a holomorphic $n$-form, and $D$ be the effective divisor defined by $\Omega$, i.e. $D={\rm div} (\Omega)$.  The restriction of $\Omega$ on $X\backslash D$ is no-where vanishing,  and thus $\mathcal{K}_{X\backslash D}$ is trivial, i.e.  $X\backslash D$ is a quasi-projective Calabi-Yau manifold. A submanifold $L $ of $X\backslash D$ is called a generalized special lagrangian submanifold respecting  to $\Omega$ and a K\"{a}hler metric $\omega$,  if $\dim_{\mathbb{R}}L=n$, $$ \omega|_{L}\equiv 0,  \  \  {\rm and} \  \ {\rm Im}(\Omega)|_{L}\equiv 0.$$
This notion of generalized special lagrangian submanifold is standard in the case of non-Ricci flat metric (cf. \cite{GHJ,ruan3}).   The real part ${\rm Re}(\Omega)$ is not a calibration respecting to the K\"{a}hler metric $\omega$, but to a non-K\"{a}hler  Hermitian metric  $\rho\omega$ by Section 10.5 in  \cite{GHJ}, where $\rho>0$ is a function defined by $\rho^{n}\omega^{n}=\frac{n!}{2^{n}}(-1)^{\frac{n^{2}}{2}} \Omega \wedge \overline{\Omega}$.

As an application of Theorem \ref{main}, we have the following theorem.

\begin{theorem}\label{main2} Let $\pi: \mathcal{X}\rightarrow \Delta$  and  $\omega_{t}$ be the same as in Theorem \ref{main}.  Assume that there is a
  zero dimensional    $X_{0,I}$. If  $\Omega_{t}$ is a section of $\mathcal{K}_{\mathcal{X}/\Delta}$ such that $D= {\rm div}(\Omega_{t})$ does not intersect with  $X_{0,I}$, then  there is a generalized special lagrangian torus  $L_{t}\subset  ( X_{t} \backslash D_{t})$ respecting  to  $ \omega_{t}$ and $e^{\sqrt{-1}\vartheta_{t}}\Omega_{t}|_{X_{t}}$ for any  $0< |t|\ll 1$ and a  phase $\vartheta_{t} \in \mathbb{R} $,   where  $D_{t}= D\cap X_{t}$.
 \end{theorem}

\section{Proof of Theorem  \ref{main}}

   \subsection{Background metric}
   In this section, we use the construction in  \cite{ruan1} to obtain some approximation    background K\"{a}hler  metrics, which are  uniformly equivalent to K\"{a}hler-Einstein metrics.

   Let $\pi: \mathcal{X}\rightarrow \Delta$ be a simple  toroidal canonical  polarized  degeneration of projective   $n$-manifolds such  that $\mathcal{X}$ is $\mathbb{Q}$-factorial.
 Since  $\mathcal{K}_{\mathcal{X}/ \Delta}$ is relative  ample, there is an embedding $\Phi: \mathcal{X}\hookrightarrow \mathbb{CP}^{N_{m}}\times \Delta$ for two integers $m>0$ and $N_{m}>0$ such that $\mathcal{K}_{\mathcal{X}/ \Delta}^{m}\cong \Phi^{*}\mathcal{O}_{\mathbb{CP}^{N_{m}}}(1)$.  There are sections $\Psi_{0}, \cdots, \Psi_{N_{m}}$ of $\mathcal{K}_{\mathcal{X}/ \Delta}^{m}$ such that, by abusing notions,  $h_{FS}=(\sum\limits_{k=0}^{N_{m}}|\Psi_{k}|^{2})^{-\frac{1}{m}}$ is the Hermitian metric whose curvature is the Fubini-Study metric, i.e.    \begin{equation}\label{e2.1} \omega^{o}=\Phi^{*}(\frac{1}{m}\omega_{FS}+\sqrt{-1}dt\wedge d\bar{t})=\sqrt{-1}\partial\overline{\partial} \log (\sum_{k=0}^{N_{m}}|\Psi_{k}|^{2})^{\frac{1}{m}}.  \end{equation}
 By regarding volume forms as Hermitian metrics of the anti-canonical bundle, we obtain a volume form $V= (\sum\limits_{k=0}^{N_{m}}|\Psi_{k}|^{2})^{\frac{1}{m}}$ on $\mathcal{X}$.
  For any $t\in \Delta^{*}$,   $  V_{t}=V\otimes (dt\wedge d\overline{t})^{-1}$  is a smooth volume form   on  $X_{t}$,  and let
   \begin{equation}\label{e2.2} \omega_{t}^{o}=\omega^{o}|_{X_{t}}=\sqrt{-1}\partial\overline{\partial} \log V_{t} .\end{equation}

Since $\mathcal{X}$ is $\mathbb{Q}$-factorial, there is a $\mu\in \mathbb{N}$ such that all of  $\mu X_{0,i}$, $i=1, \cdots, l$,  are Cartier divisors.
Let $\|\cdot \|_{i}$ be a smooth  Hermitian metric of  $\mathcal{O}(\mu X_{0,i})$ on $\mathcal{X}$, and $s_{i}$ be a defining  section of $\mu X_{0,i}$, i.e. ${\rm div}(s_{i})=\mu X_{0,i}$. Here the  Hermitian metric $\|\cdot \|_{i}$ being smooth means that   $\|\cdot \|_{i}$ is locally given by the restriction of  a smooth positive function $\varrho$ on  the ambient space $\mathbb{C}^{\nu}$ for a local embedding of an open subset $U$ of $\mathcal{X}$ into $\mathbb{C}^{\nu}$, and a  trivialization of  $\mathcal{O}(\mu X_{0,i})$ on $ U$. In this case, ${\rm Ric}(\|\cdot\|_{i})$ is the restriction of the smooth form $-\sqrt{-1}\partial\overline{\partial}\log \varrho$ on $\mathbb{C}^{\nu}$.

We assume that $s_{1}\cdot\cdots s_{l}=t^{\mu}$ by choosing the parameter $t\in \Delta$ appropriately.
Let  \begin{equation}\label{e2.3}  \alpha_{i}=\frac{1}{\mu}\log \| s_{i} \|^{2}_{i},   \  \   \    \   \     \chi_{t}=(\log|t|^{2})^{2}\prod_{i=1}^{l}\alpha_{i}^{-2},   \end{equation}
  and
 \begin{equation}\label{e2.4}\begin{split}
\tilde{\omega}_{t} & = \sqrt{-1}\partial\overline{\partial}\log \chi_{t}  V_{t}\\ & =   \omega^{o}_{t} +\sqrt{-1}\partial\overline{\partial} \log \chi_{t}\\ & =  \omega^{o}_{t}+2\sum_{i=1}^{l}(\frac{{\rm Ric}(\|\cdot\|_{i})}{\alpha_{i}}+\sqrt{-1}\frac{\partial \alpha_{i} \wedge \overline{\partial}\alpha_{i}}{\alpha_{i}^{2}})|_{X_{t}}
\end{split}\end{equation} on $X_{t}$ for $t\neq 0$.
We can assume that   $\|s_{i} \|_{i}\leq\varepsilon \ll 1$      such that   $$\frac{1}{2}\omega^{o}\leq \omega^{o}+\sum\limits_{i=1}^{l}\frac{2}{\alpha_{i}}{\rm Ric}(\|\cdot\|_{i}) \leq 2 \omega ^{o}$$ on $ \mathcal{X}\backslash X_{0} $ by multiplying certain constants if necessary.
We denote $X_{0,I}^{o}=X_{0,I}\backslash \bigcup\limits_{i \notin I}X_{0,i}  $, and  define a complete  K\"{a}hler metric  \begin{equation}\label{e2.5++} \tilde{\omega}_{0,I}   = \omega^{o}|_{X_{0,I}^{o}}+2\sum\limits_{i \notin I}(\frac{{\rm Ric}(\|\cdot\|_{i})}{\alpha_{i}}+\sqrt{-1}\frac{\partial \alpha_{i} \wedge \overline{\partial}\alpha_{i}}{\alpha_{i}^{2}})|_{X_{0,I}^{o}} \end{equation} on $X_{0,I}^{o}$.

The K\"{a}hler metric $\tilde{\omega}_{t}$ is the background metric we need.
 Note that our assumption of $\mathcal{X}$ is stronger than the one in  \cite{ruan1}, and however is weaker than that in  \cite{ruan2}. Nevertheless,   the arguments in Section 3 of \cite{ruan1} and Section 4 of  \cite{ruan2} show that the curvature of $\tilde{\omega}_{t}$ and the Ricci potential $\log (\frac{V_{t}}{\tilde{\omega}_{t}^{n}}) $ are  bounded independent of $t$, which can also be obtained by the calculation in Section 3.2.   Thus we have the $C^{0}$ and $C^{2}$ estimates for the potential function of the K\"{a}hler-Einstein metric by the standard estimates for Monge-Amp\`{e}re equations (cf. \cite{aubin, Yau1}).

\begin{proposition}\label{prop}
    Let $\varphi_{t}$ be  the  unique solution of Monge-Amp\`{e}re equation  \begin{equation}\label{mae}(\tilde{\omega}_{t}+\sqrt{-1}\partial\overline{\partial}\varphi_{t})^{n}=e^{\varphi_{t}}\chi_{t}V_{t}, \  \  \end{equation}   and  $  \omega_{t}=\tilde{\omega}_{t}+\sqrt{-1}\partial\overline{\partial}\varphi_{t}$ be the K\"{a}hler-Einstein metric on $X_{t}$.
     Then $$|\varphi_{t}|\leq C_{1},  \  \  {\rm  and }  \  \   C_{2}^{-1}\tilde{\omega}_{t}\leq \omega_{t} \leq C_{2} \tilde{\omega}_{t},$$ for  constants $C_{1}>0$ and $C_{2}>0$ independent of $t$.
  \end{proposition}

Once Proposition \ref{prop} is obtained, \cite{Tian,leung,ruan1,ruan2}  prove the convergence of $\omega_{t}$ to a complete K\"{a}hler-Einstein metric $\omega_{0}$ on the regular locus $X_{0,reg}$ in the Cheeger-Gromov sense, i.e. for any smooth family of embeddings $F_{t}: X_{0,reg} \rightarrow X_{t}$, $F_{t}^{*} \omega_{t} $ converges to $\omega_{0}$ in the locally $C^{\infty}$-sense when $t\rightarrow 0$.  When $\sharp I=1$, $ \tilde{\omega}_{0,I}$ is uniformly  equivalent to the K\"{a}hler-Einstein metric  $\omega_{0}$ on $X_{0,I}^{o}\subset X_{0,reg}$.

    \subsection{Proof of Theorem  \ref{main}}
    Now we study the local collapsing behaviour of    K\"{a}hler-Einstein metrics $\omega_{t}$.

For a point $p\in X_{0,I}$, let  $U\subset \mathcal{X}$ be a neighborhood of $p$   isomorphic to an open subset of a toric variety  $ \mathcal{Y}_{\sigma}$, denoted still by $U$, such that $U\cap  X_{0,I'}$ is empty  for any $I'   \nsubseteq   I=\{1,\cdots, s+1\}$.  Since $\mathcal{X}$ is $\mathbb{Q}$-factorial, so is $ \mathcal{Y}_{\sigma}$, and $ \mathcal{Y}_{\sigma}$ has only orbifold singularities, which is equivalent to the  rational    cone $\sigma$ being simplicial (cf. \cite{Fulton}).

      If $v_{0}, \cdots, v_{s}\in N$ are primitive vectors  belonging to $1$-dimensional faces  and   generating  $\sigma$ in $N_{\mathbb{R}}$, we denote $N_{\sigma}'={\rm Span}_{\mathbb{Z}}\{v_{0}, \cdots, v_{s}\}$ which is a sublattice of $N_{\sigma}=\mathbb{Z}\cdot (\sigma\cap N)$, and $M(\sigma)=\sigma^{\bot} \cap M\cong \mathbb{Z}^{n-s}$. Then $M\cong M_{\sigma}\oplus M(\sigma)$ where $M_{\sigma}=\hom_{\mathbb{Z}}(N_{\sigma}, \mathbb{Z})\cong M/M(\sigma)$, and  $M_{\sigma}$ is a sublattice of  $M_{\sigma}'=\hom_{\mathbb{Z}}  (N_{\sigma}',\mathbb{Z})={\rm Span}_{\mathbb{Z}}(v_{0}^{*}, \cdots, v_{s}^{*})$, where $v_{j}^{*}$ is the dual vector of $v_{j}$.  Note that the restriction of $\pi$ on $U$   is given by a monomial $\mathcal{Z}^{ u_{\sigma}}$, where
       $u_{\sigma} \in \check{\sigma} \cap M_{\sigma}$ satisfies   $\langle u_{\sigma}, v_{j}\rangle =1$ for $j=0, \cdots, s$, i.e.  $u_{\sigma}=\sum\limits_{j=0}\limits^{s}v_{j}^{*}$.

  If $G= N_{\sigma}/ N_{\sigma}'$, and $\mathcal{Y}'_{\sigma}={\rm Spec}(\mathbb{C}[\check{\sigma}\cap M'_{\sigma}]) \cong \mathbb{C}^{s+1}$, then the finite group  $G$ acts on $\mathcal{Y}'_{\sigma}$ by $v \cdot \mathcal{Z}^{u}=\exp (2\pi \sqrt{-1}\langle v, u\rangle)\cdot \mathcal{Z}^{u}$ for any $v\in N_{\sigma}$ and $u\in M_{\sigma}'$, and $\mathcal{Y}'_{\sigma}/G \times (\mathbb{C}^{*})^{n-s}\cong \mathcal{Y}_{\sigma}$. We denote  $q_{\sigma}: \mathcal{Y}'_{\sigma} \times (\mathbb{C}^{*})^{n-s}\rightarrow  \mathcal{Y}_{\sigma}$ the quotient map of the $G$-action.
    Let $z_{j}=\mathcal{Z}^{v_{j}^{*}}$, $j=0, \cdots, s$, be   coordinates on $\mathcal{Y}'_{\sigma}$, and $z_{s+1}, \cdots, z_{n}$ be  coordinates on $(\mathbb{C}^{*})^{n-s}$. The restriction  $q_{\sigma}: T_{N_{\sigma}'} \times (\mathbb{C}^{*})^{n-s}\rightarrow  T_{N}$ is a finite covering map, where  $T_{N}=N\otimes_{\mathbb{Z}}\mathbb{C}^{*}$ and $T_{N_{\sigma}'}=N_{\sigma}'\otimes_{\mathbb{Z}}\mathbb{C}^{*}$.

If we denote
    $Y_{\sigma,t}={\rm div}(\mathcal{Z}^{u_{\sigma}}-t) $, $t\in\mathbb{C}$, then  $Y_{\sigma,t}\cap U=X_{t}\cap U$,  and  $Y_{\sigma,0}=\sum\limits_{j=0}^{s}D_{j}$ where $D_{j}$  is a primitive toric Weil divisor of $Y_{\sigma}$. The restriction  $q_{\sigma}: q_{\sigma}^{-1}(Y_{\sigma,t})\rightarrow Y_{\sigma,t}$ is a finite covering map as  $X_{t}\cap U \subset T_{N}$ when $t\neq0$, and
    $q_{\sigma}^{-1}(Y_{\sigma,t})$
    is given by the equation
    $z_{0} \cdot\cdots z_{s}=t$ in $\mathcal{Y}'_{\sigma} \times (\mathbb{C}^{*})^{n-s}$. We can regard $z_{1}, \cdots, z_{n}$ as coordinates of $q_{\sigma}^{-1}(Y_{\sigma,t})$ for any $t\neq 0$. We assume that  $U\subset \mathcal{Y}_{\sigma}$  satisfies  $q_{\sigma}^{-1}(U)=\{(z_{0}, \cdots,z_{s})\in \mathcal{Y}'_{\sigma} ||z_{j}|<\epsilon, 0\leq j \leq s\}\times (U\cap X_{0,I})$ for an  $\epsilon<1$ without loss  of generality.

    Let $ x_{0}, \cdots, x_{s}$ be coordinates on $N'_{\sigma,\mathbb{R}}=N'_{\sigma}\otimes_{\mathbb{Z}}\mathbb{R} \cong \mathbb{R}^{s+1}$ respecting to the basis $v_{0}, \cdots, v_{s}$. Note that the interior of the $s$-dimensional  cell $\overline{B}_{I} \in \mathcal{B}$  associated  to $X_{0,I}$  is given by
     \begin{equation}\label{e3.01}\begin{split} B_{I}& =\{v\in {\rm int} (\sigma)  | \langle v, u_{\sigma}\rangle =1\}\\ & =\{(x_{0}, \cdots, x_{s})\in \mathbb{R}^{s+1} | \sum_{j=0}^{s}x_{j} =1, x_{j}>0, j=0, \cdots, s\}\\ &=\{(x_{1}, \cdots, x_{s})\in \mathbb{R}^{s} |   \sum_{j=1}^{s}x_{j} <1,  x_{j}>0, j=1, \cdots, s\}. \end{split} \end{equation} Here  we regard $x_{1}, \cdots, x_{s}$ as coordinates on  $B_{I}$.

 For any $t\in\Delta^{*}$,  we define the   covering map  \begin{equation}\label{e3.1} P_{t}: \mathbb{C}^{s}\times (\mathbb{C}^{*})^{n-s} \rightarrow q_{\sigma}^{-1}( Y_{\sigma,t})  \  \    \end{equation} by letting $z_{j}=e^{(\log |t| )w_{j}}$   and $x_{j}= {\rm Re} (w_{j}) $,   $j=1, \cdots, s$,  i.e.  $$P_{t}(w_{1}, \cdots, w_{s}, z_{s+1}, \cdots, z_{n})=(e^{(\log |t| )w_{1}}, \cdots,e^{(\log |t| ) w_{s}}, z_{s+1}, \cdots, z_{n}).$$  The fundamental domains  of $P_{t}$ are \begin{equation}\label{e3.010} \mathfrak{D}_{t,\nu}=\{(w_{1}, \cdots, w_{s})\in \mathbb{C}^{s}| \frac{2\pi \nu}{\log |t| } \leq {\rm Im}(w_{j}) \leq  \frac{2\pi (\nu+1)}{\log |t| }\}\times (\mathbb{C}^{*})^{n-s}\end{equation}  for $\nu \in\mathbb{Z}$, and naturally      $ (\mathbb{C}^{s}/\sqrt{-1}\frac{2 \pi \mathbb{Z}^{s}}{\log |t|} )\times (\mathbb{C}^{*})^{n-s} $ is biholomorphic to $ q_{\sigma}^{-1}( Y_{\sigma,t}) $ by further  setting $ z_{0}=t z_{1}^{-1}\cdots z_{s}^{-1}= t\exp (-\sum\limits_{j=1}^{s}(\log |t| )w_{j}).$

  Note that if $|z_{j}|<\epsilon$, $j=0, \cdots, s$,  then $x_{j}> \frac{\log \epsilon}{\log |t|}$ for $j=1, \cdots, s$, and $\log |t|(1- \sum\limits_{j=1}^{s}x_{j}) = \log |z_{0}|<\log \epsilon, $ which implies $$ P^{-1}_{t}(q_{\sigma}^{-1}( Y_{\sigma,t}\cap U))= B_{t}\times\sqrt{-1}\mathbb{R}^{s} \times (\mathbb{C}^{*})^{n-s}, $$
   where $$ B_{t}=\{ (x_{1}, \cdots, x_{s})\in \mathbb{R}^{s} | x_{j}> \frac{\log \epsilon}{\log |t|}, j=1, \cdots, s,  1- \sum\limits_{j=1}^{s}x_{j} > \frac{\log \epsilon}{\log |t|}\}\subset B_{I}.$$  Hence $$q_{\sigma}^{-1}( Y_{\sigma,t}\cap U)\subset B_{I}\times\sqrt{-1}(\mathbb{R}^{s}/\frac{2 \pi \mathbb{Z}^{s}}{\log |t|} ) \times (\mathbb{C}^{*})^{n-s}\subset q_{\sigma}^{-1}( Y_{\sigma,t}).   $$

   \begin{lemma}\label{le1} Let $K\subset B_{I}$ be a compact subset such that $K \subset B_{t} $ for $|t|\ll 1$.
 On  $K\times \sqrt{ -1}\mathbb{R}^{s} \times (U\cap X_{0,I}) \subset (q_{\sigma}\circ  P_{t})^{-1}(U\cap Y_{\sigma, t}) $,  when $t\rightarrow 0$,
  \begin{itemize}
  \item[i)]    $$ P_{t}^{*}q_{\sigma}^{*} \chi_{t} V_{t}\rightarrow V_{0}'= \frac{1}{4(1-\sum\limits_{j=1}^{s}x_{j})^{2} }\prod_{j=1}^{s}\frac{dw_{j}\wedge d\bar{w}_{j}}{4 x_{j}^{2} }\wedge V_{I},  $$in the $C^{\infty}$-sense, where  $ V_{I}$ is a smooth volume form on $U\cap X_{0,I}$.
    \item[ii)]
   $$P_{t}^{*}q_{\sigma}^{*} \tilde{\omega}_{t}\rightarrow  \omega^{o}_{U,I}+\frac{\sqrt{-1}}{2}( \sum_{j=1}^{s}  \frac{ dw_{j}\wedge  d\bar{w}_{j}}{x_{j}^{2}}+ \frac{\sum\limits_{ij=1}^{s}   dw_{i}\wedge  d\bar{w}_{j}}{(1-\sum\limits_{j=1}^{s}x_{j})^{2} } )= \tilde{\omega}^{o} $$in the $C^{\infty}$-sense, where $\omega^{o}_{U,I}$ is the pull-back of  the complete   K\"{a}hler metric $\tilde{\omega}_{0,I} $  on $U\cap X_{0,I}$.
   \end{itemize}
 \end{lemma}

 \begin{proof} Let  $ w_{0}=1+\sqrt{-1}\frac{\arg( t)}{\log |t|}-w_{1}- \cdots -w_{s}$  on $(q_{\sigma}\circ  P_{t})^{-1}( Y_{\sigma, t})$,  and $x_{0}=1-x_{1}- \cdots -x_{s}$ on $B_{I}$. We have $z_{0}=e^{\log|t|w_{0}}$ and $dw_{0}=-dw_{1}- \cdots -dw_{s}$ on $(q_{\sigma}\circ  P_{t})^{-1}( Y_{\sigma, t})$.

  Now, we claim that for a smooth function $\lambda$ on $\mathcal{Y}'_{\sigma}\times (\mathbb{C}^{*})^{n-s}$,   $\lambda \circ P_{t}  \rightarrow \lambda'=\lambda(0,z_{s+1},\cdots,z_{n})$ and $dz_{j}=\frac{\partial z_{j}}{\partial w_{j}}dw_{j}\rightarrow 0$, $j=0, \cdots, s$,   in the $C^{\infty}$-sense on any compact subset of  $ (q_{\sigma}\circ  P_{t})^{-1}( U \cap Y_{\sigma, t})$, when $t\rightarrow 0$. Since $$ |\frac{\partial^{k} z_{j}}{\partial w_{j}^{k}}|=|(\log|t|)^{k} e^{ (\log|t|) x_{j}}| \leq |(\log|t|)^{k} |e^{ \varepsilon_{j} \log|t| } \rightarrow 0, \  \  \ |\frac{\partial z_{0}}{\partial w_{j}}|=|\frac{\partial z_{0}}{\partial w_{0}}|$$ for a  $\varepsilon_{j}>0$, $0\leq j \leq s$, the claim follows by $|\frac{\partial^{k} \lambda}{\partial z_{i_{1}}^{k_{1}} \cdots \partial z_{i_{s'}}^{k_{s'}}}|\leq C$ for some constants $C>0$.

 Since $\mathcal{Y}_{\sigma}$ has only Gorenstein orbifold singularities, for the generator $\Omega_{\sigma}\in \mathcal{O}(\mathcal{K}_{\mathcal{Y}_{\sigma}})$, $q_{\sigma}^{*}\Omega_{\sigma}$ is a $G$-invariant  no-where vanishing holomorphic $(n+1,0)$-form on $\mathcal{Y}'_{\sigma}\times (\mathbb{C}^{*})^{n-s}$, and thus $$q_{\sigma}^{*}V= \eta \prod_{j=0}^{n}dz_{j}\wedge d\bar{z}_{j},$$ where $\eta>0$ is a smooth function on $q_{\sigma}^{-1}(U)$.  We obtain $$ q_{\sigma}^{*} V_{t}=\eta \prod_{j=1}^{s}\frac{dz_{j}\wedge d\bar{z}_{j}}{|z_{j}|^{2}}\wedge \prod_{i=s+1}^{n}dz_{i}\wedge d\bar{z}_{i}$$ on $q_{\sigma}^{-1}(X_{t} \cap U)$.

 Without loss of generality, we assume that $I=\{1, \cdots, s+1\}$.
Under a   local  trivialization of $\mathcal{O}(\mu X_{0,i})$, $i\in I$,  on $U$, we have that  $q_{\sigma}^{*}s_{i}=z_{j}^{\mu}$, where $j=i-1$,   and   the Hermitian metric  $\|\cdot \|_{i}$ is a given by restricting  a smooth function $\rho_{j}'$ on an open subset $\mathbb{C}^{\nu}$ for  a local  embedding $U\hookrightarrow \mathbb{C}^{\nu}$.   Thus
    $q_{\sigma}^{*} \alpha_{j+1}=\log \rho_{j} | z_{j} |^{2}$ for  $j=0, \cdots, s$,   where $\rho_{j}=\rho_{j}'\circ q_{\sigma}>0$ are smooth function on  $q_{\sigma}^{-1}(U)$, and $q_{\sigma}^{*} \alpha_{i}<0$,  $i=s+2, \cdots, l$,  are also smooth  functions.  By (\ref{e2.3}),  $$ q_{\sigma}^{*} \chi_{t} V_{t}=\eta''\frac{(\log|t|^{2})^{2}}{(\log (\rho_{0} | z_{0} |^{2}))^{2}}\prod_{j=1}^{s}\frac{dz_{j}\wedge d\bar{z}_{j}}{|z_{j}|^{2}(\log (\rho_{j} | z_{j} |^{2}))^{2}}\wedge \prod_{i=s+1}^{n}dz_{i}\wedge d\bar{z}_{i},$$ where $\eta''>0$ is a smooth function on $ q_{\sigma}^{-1}(U)$,  and  $$ P_{t}^{*}q_{\sigma}^{*} \chi_{t} V_{t}=\frac{\eta'' \circ P_{t}}{(\frac{\log \rho_{0}}{\log|t|^{2}} +2x_{0} )^{2}}\prod_{j=1}^{s}\frac{dw_{j}\wedge d\bar{w}_{j}}{(\frac{\log \rho_{j}}{\log|t|^{2}} +2x_{j} )^{2}}\wedge\prod_{i=s+1}^{n}dz_{i}\wedge d\bar{z}_{i}.$$  By taking $t\rightarrow 0$, we obtain the convergence of volume forms.

   We have
 $q_{\sigma}^{*}\omega^{o}$ is a smooth $(1,1)$-form on $q_{\sigma}^{-1}(U)$, and $$P_{t}^{*} q_{\sigma}^{*}\omega^{o}_{t}=\sqrt{-1}P_{t}^{*} q_{\sigma}^{*}\partial\overline{\partial} \log V_{t}= \sqrt{-1}\partial\overline{\partial} \log \eta \rightarrow \sqrt{-1}\partial\overline{\partial} \log \eta',$$  in the $C^{\infty}$-sense, when $t\rightarrow 0$,  where $\eta'=\eta(0,z_{s+1}, \cdots, z_{n})>0$. Note that  $\sqrt{-1}\partial\overline{\partial} \log \eta'$ is the pull-back of $\omega^{o}|_{X_{0,I}\cap U} $.     Since $q_{\sigma}^{*} \frac{{\rm Ric}(\|\cdot\|_{i})}{\alpha_{i}}$ and $q_{\sigma}^{*} \frac{\partial \alpha_{i} \wedge \overline{\partial}\alpha_{i}}{\alpha_{i}^{2}}$, $i=s+2, \cdots,l$, are also smooth $(1,1)$-forms on $q_{\sigma}^{-1}(U)$, we have $$ P_{t}^{*}q_{\sigma}^{*} (\frac{{\rm Ric}(\|\cdot\|_{i})}{\alpha_{i}}+\sqrt{-1}\frac{\partial \alpha_{i} \wedge \overline{\partial}\alpha_{i}}{\alpha_{i}^{2}}) \rightarrow \beta_{i},$$ in the $C^{\infty}$-sense, where $\beta_{i}$ is the pull-back of the  smooth K\"{a}hler form $(\frac{{\rm Ric}(\|\cdot\|_{i})}{\alpha_{i}}+\sqrt{-1}\frac{\partial \alpha_{i} \wedge \overline{\partial}\alpha_{i}}{\alpha_{i}^{2}})|_{U\cap X_{0,I}}$ on $U\cap X_{0,I}$.  Thus $$\omega^{o}_{U,I}=\sqrt{-1}\partial\overline{\partial} \log \eta'+2\sum\limits_{i=s+2}^{l}\beta_{i}$$ is the pull-back of the restriction of   $\tilde{\omega}_{0,I} $  on $U\cap X_{0,I}$ by (\ref{e2.5++}).

 On $K$, $(\log |t|)x_{j} \rightarrow - \infty$, $j=0, \cdots, s$,  and thus,
  $$P_{t}^{*} q_{\sigma}^{*} \frac{{\rm Ric}(\|\cdot\|_{j+1}) }{\alpha_{j+1}}=\frac{\sqrt{-1} \partial\overline{\partial} \log \rho_{i}}{ \log \rho_{i}+2 (\log |t|)x_{i}}\rightarrow 0,  $$ in the $C^{\infty}$-sense. Furthermore,
$$P_{t}^{*} q_{\sigma}^{*} \frac{\partial \alpha_{j+1} \wedge \overline{\partial}\alpha_{j+1}}{\alpha_{j+1}^{2}}=\frac{(\partial \log \rho_{j}+ \log|t| dw_{j})\wedge  (\overline{\partial} \log \rho_{j}+ \log|t| d\bar{w}_{j})}{(\log \rho_{j}+2(\log|t|)x_{j})^{2}}\rightarrow \frac{ dw_{j}\wedge  d\bar{w}_{j}}{4x_{j}^{2}},$$ in the $C^{\infty}$-sense, when $t\rightarrow 0$.
Thus we obtain the conclusion by (\ref{e2.4}), and $$\frac{ dw_{0}\wedge  d\bar{w}_{0}}{4x_{0}^{2}}=  \frac{\sum\limits_{ij=1}^{s}   dw_{i}\wedge  d\bar{w}_{j}}{4(1-\sum\limits_{j=1}^{s}x_{j})^{2} } .  $$
 \end{proof}

  \begin{lemma}\label{le2} Let $\varphi_{t}$ be the unique solution of (\ref{mae}),    and  $  \omega_{t}=\tilde{\omega}_{t}+\sqrt{-1}\partial\overline{\partial}\varphi_{t}$.
   For any sequence $t_{k} \rightarrow 0$, a subsequence of  $\varphi_{t_{k}}\circ q_{\sigma} \circ P_{t_{k}}$ converges to  $\varphi_{0}$ in the $C^{\infty}$-sense on $K\times \sqrt{ -1}\mathbb{R}^{s} \times (U\cap X_{0,I})$, where $\varphi_{0}$ is  a  smooth function on $B_{I}\times \sqrt{ -1}\mathbb{R}^{s} \times (U\cap X_{0,I})$  satisfying  the complex  Monge-Amp\`{e}re equation  \begin{equation}\label{mae2} (\tilde{\omega}^{o}+\sqrt{-1}\partial\overline{\partial}\varphi_{0})^{n}=e^{\varphi_{0}}V_{0}', \  \    \end{equation} $${\rm  with}  \  \ |\varphi_{0}|\leq C_{3}, \  \  {\rm and} \  \ C_{4}^{-1}\tilde{\omega}^{o}\leq \tilde{\omega}^{o}+\sqrt{-1}\partial\overline{\partial}\varphi_{0}  \leq C_{4}\tilde{\omega}^{o}.   $$  Furthermore, $\varphi_{0}$ is independent of ${\rm Im}(w_{j})$, $j=1, \cdots, s$, i.e.   $$\varphi_{0}=\varphi_{0}(x_{1}, \cdots x_{s},z_{s+1},\cdots,z_{n}).$$
  \end{lemma}

 \begin{proof} By Proposition \ref{prop},
 we have  that $$|\varphi_{t}|\leq C, \  \  {\rm  and }  \  \  C^{-1}P_{t}^{*}q_{\sigma}^{*} \tilde{\omega}_{t}\leq P_{t}^{*}q_{\sigma}^{*}(\tilde{\omega}_{t}+\sqrt{-1}\partial\overline{\partial}\varphi_{t})\leq CP_{t}^{*}q_{\sigma}^{*} \tilde{\omega}_{t}$$ for a constant $C>0$.
 We obtain the $C^{2,\alpha}$-estimates for $\varphi_{t}$, i.e. $\|\varphi_{t}\circ q_{\sigma}\circ P_{t}\|_{C^{2,\alpha}}\leq \bar{C}$,  by Lemma \ref{le1} and
 the
Evans-Krylov theory (cf.  \cite{GT, Siu}), and    the higher order estimates  $\|\varphi_{t}\circ q_{\sigma}\circ P_{t}\|_{C^{\nu}}\leq C(\nu)$ by  the standard
 Schauder
estimates on any compact subset $K' \subset K\times \sqrt{ -1}\mathbb{R}^{s} \times (U\cap X_{0,I})$.  Thus by passing to a subsequence of $t_{k}$,  $\varphi_{t_{k}}\circ q_{\sigma} \circ P_{t_{k}}$ converges to a smooth function   $\varphi_{0}$ in the locally  $C^{\infty}$-sense, and $\varphi_{0}$ satisfies the the complex  Monge-Amp\`{e}re equation (\ref{mae2}) by Lemma \ref{le1}.

 Since
 $ \varphi_{t}\circ q_{\sigma}\circ P_{t} $ is a periodic function with period $\sqrt{-1} \frac{2 \pi \mathbb{Z}^{s}}{\log |t|}$, i.e. $$\varphi_{t}\circ q_{\sigma}\circ P_{t} (\mathrm{w}, \mathrm{z})=\varphi_{t}\circ q_{\sigma}\circ P_{t}(\mathrm{w}+\sqrt{-1}\frac{2\pi \mathfrak{m}}{\log |t|},\mathrm{z}),$$ for any  $\mathfrak{m} \in \mathbb{Z}^{s}$,  where $\mathrm{w}=(w_{1},\cdots,w_{s})$ and $\mathrm{z}=(z_{s+1}, \cdots, z_{n})$, we obtain that     $\varphi_{0}$ is independent of ${\rm Im}(w_{j})$, $j=1, \cdots, s$, by the smooth convergence.
  \end{proof}

   Since $\frac{\partial^{2} \varphi_{0}}{\partial w_{i} \overline{\partial}w_{j}}= \frac{\partial^{2} \varphi_{0}}{4 \partial x_{i}\partial x_{j}}$, the corresponding  Riemannian metric of $\tilde{\omega}^{o}+\sqrt{-1}\partial\overline{\partial}\varphi_{0}$ is  \begin{equation}\label{e3.50} g_{0}= \sum_{ij=1}^{s} (\frac{\delta_{ij}}{x_{i}^{2}}+\frac{1}{(1-\sum\limits_{j=1}^{s}x_{j})^{2}}+\frac{\partial^{2} \varphi_{0}}{2 \partial x_{i}\partial x_{j}})(dx_{i}dx_{j}+d\theta_{i}d\theta_{j})+\mathcal{G}_{0},
  \end{equation} where $\theta_{j}={\rm Im} (w_{j})$, $j=1, \cdots, n$, and $\mathcal{G}_{0}$ denotes the remaining   terms that do  not involve any $d\theta_{i}d \theta_{j}$ and $dx_{i}d x_{j}$.

  Note that both  $\tilde{\omega}^{o}$ and $\varphi_{0}$ are invariant under the translation $w_{j}\mapsto w_{j}+\lambda_{j} \sqrt{-1}$, $j=1,\cdots,s$, for any $(\lambda_{1}, \cdots, \lambda_{s})\in \mathbb{R}^{s}$.  Hence for any $t\neq 0$,  $\tilde{\omega}^{o}+\sqrt{-1}\partial\overline{\partial}\varphi_{0}$ descents to  a K\"{a}hler  metric $\omega_{t}^{sf}$ on $Y_{\sigma,t} \cap U$, which satisfies   that  \begin{equation}\label{e3.5} P_{t}^{*}q_{\sigma}^{*} \omega_{t}^{sf} = \tilde{\omega}^{o}+\sqrt{-1}\partial\overline{\partial}\varphi_{0},  \  \  {\rm and} \  \   \|  \omega_{t_{k}}- \omega_{t_{k}}^{sf} \|_{C_{loc}^{\nu}(Y_{\sigma,t_{k}} \cap U, \omega_{t_{k}}^{sf})} \rightarrow 0,
 \end{equation} for any $\nu>0$,  when $t_{k}\rightarrow 0$ by Lemma \ref{le2}.

Define a fibration
$$ \tilde{f}_{t}: B_{I} \times \sqrt{-1}(\mathbb{R}^{s}/ (\frac{2\pi \mathbb{Z}^{s}}{\log |t|}))\times (\mathbb{C}^{*})^{n-s}\rightarrow   B_{I}\times (\mathbb{C}^{*})^{n-s}$$ by the projection. Note that $ \tilde{f}_{t}$ is $G$-equivariant, $ \tilde{f}_{t}$ induces a $T^{s}$-fibration  \begin{equation}\label{e3.01}f_{t}:U\cap Y_{\sigma, t} \rightarrow B_{t}\times (\mathbb{C}^{*})^{n-s}, \  \ {\rm  with}  \  \    \tilde{f}_{t}=f_{t} \circ q_{\sigma}. \end{equation}
 For a point  $(\mathrm{x}, \mathrm{z})\in  B_{t}\times (\mathbb{C}^{*})^{n-s}$, where $\mathrm{x}=(x_{1}, \cdots, x_{s})\in B_{t}$ and $\mathrm{z}=(z_{s+1}, \cdots, z_{n})\in (\mathbb{C}^{*})^{n-s}$,  the fiber $f_{t}^{-1}(\mathrm{x}, \mathrm{z})$ satisfies that $(q_{\sigma}\circ P_{t})^{-1} (f_{t}^{-1}(\mathrm{x}, \mathrm{z}))=\{(\mathrm{x}+\sqrt{-1}\theta, \mathrm{z})|\theta =(\theta_{1}, \cdots, \theta_{s})\in \mathbb{R}^{s} \}$.   Hence the restriction of the K\"{a}hler metric $\omega_{t}^{sf}$ on $f_{t}^{-1}(\mathrm{x}, \mathrm{z})$ is a flat Riemannian metric, i.e. $\omega_{t}^{sf}$ is a semi-flat metric,  and $$     {\rm diam}_{\omega_{t_{k}}}(f_{t_{k}}^{-1}(\mathrm{x}, \mathrm{z}))\sim  {\rm diam}_{\omega_{t_{k}}^{sf}}(f_{t_{k}}^{-1}(\mathrm{x}, \mathrm{z}))\leq \frac{2\pi s \sqrt{C_{\mathrm{x}, \mathrm{z}} }}{-\log |t|}  \rightarrow 0 ,$$ when $t \rightarrow 0$, by (\ref{e3.5}) and (\ref{e3.50}), where $$ C_{\mathrm{x}, \mathrm{z}}= \sum_{ij=1}^{s} \Big | \frac{\delta_{ij}}{x_{i}^{2}}+\frac{1}{(1-\sum\limits_{j=1}^{s}x_{j})^{2}}+\frac{\partial^{2} \varphi_{0}}{2 \partial x_{i}\partial x_{j}}(\mathrm{x}, \mathrm{z})\Big |.  $$

  We denote $W_{U}= B_{I}  \times (U\cap X_{0,I})$, and naturally  regard  $W_{U} \subset B_{I} \times \sqrt{-1}(\mathbb{R}^{s}/ (\frac{2\pi \mathbb{Z}^{s}}{\log |t|}))\times (U\cap X_{0,I})$ given by $\theta_{j}=0$, $j=1,\cdots, n$. We let   $g_{W_{U}}=g_{0}|_{W_{U}}$.    If  $p\in W_{U}$, and  $r>0$ such that the metric ball $B_{g_{W_{U}}}(p,r)\subset K''$ for a compact subset $K''\subset W_{U}$, then \begin{equation}\label{e3.501}(B_{\omega_{t_{k}}}(p_{t_{k}},r), \omega_{t_{k}}) \ \ {\rm and}   \  \  ( B_{\omega_{t_{k}}^{sf}}(p_{t_{k}},r),\omega_{t_{k}}^{sf}) \rightarrow (B_{g_{W_{U}}}(p,r), g_{W_{U}}) \end{equation} in the Gromov-Hausdorff sense by (\ref{e3.5}), when $t_{k} \rightarrow 0$,  for some  $p_{t} \in X_{t} \cap U$.  By  Gromov's precompactness theorem (cf.  \cite{Gromov}),  $(X_{t_{k}},  \omega_{t_{k}}, p_{t_{k}})$ converges to a complete metric space $(W, d_{W}, p_{\infty})$ of Hausdorff dimension $\varrho$  in the pointed Gromov-Hausdorff sense, and   there is a local isometric embedding $(B_{g_{W_{U}}}(p_{\infty},r), g_{W_{U}}) \hookrightarrow (W, d_{W})$, which implies   $\varrho= \dim_{\mathbb{R}} B_{I} \times X_{0,I}$.

 In summary, we have the following proposition.

  \begin{proposition}\label{prop22}
   There is  a semi-flat K\"{a}hler-Einstein metric $\omega_{t}^{sf}$ on $X_{t} \cap U$ respecting to $f_{t}$  such that $$ \|  \omega_{t_{k}}- \omega_{t_{k}}^{sf} \|_{C_{loc}^{\nu}(Y_{\sigma,t_{k}} \cap U,\omega_{t_{k}}^{sf})} \rightarrow 0,
 $$ for any $\nu>0$, and a sequence $t_{k}\rightarrow 0$. Furthermore, $(X_{t_{k}},  \omega_{t_{k}}, p_{t_{k}})$ converges to a complete metric space $(W, d_{W}, p_{\infty})$ in the pointed Gromov-Hausdorff sense by choosing  some base points   $p_{t} \in X_{t}$, and the Hausdorff dimension of $W$ equals to $\dim_{\mathbb{R}} B_{I}  + 2 \dim_{\mathbb{C}}  X_{0,I}$.
  \end{proposition}

  Now we are ready to prove  Theorem    \ref{main}.

\begin{proof}[Proof of Theorem    \ref{main}]
   Let $\{U_{\gamma}\}$ be an open  cover of $X_{0,I}^{o}$ such that any $U_{\gamma}\subset \mathcal{X}$ is isomorphic to an open subset of a toric variety, and does not intersect with  $\bigcup\limits_{i\notin I} X_{0,i}$.  By applying  the above arguments to  $U_{\gamma}$, we have  $W_{U_{\gamma}}=B_{I}  \times (U_{\gamma}\cap X_{0,I})\subset B_{I} \times \sqrt{-1}(\mathbb{R}^{s}/ (\frac{2\pi \mathbb{Z}^{s}}{\log |t|}))\times (U_{\gamma}\cap X_{0,I})$, and a metric $g_{U_{\gamma}}=g_{\gamma,0}|_{W_{U_{\gamma}}}$, where $g_{\gamma,0}$ is the  Riemannian metric given by (\ref{e3.50}).  If $\omega_{\gamma,t}^{sf}|_{X_{t}\cap U_{\gamma}}$ denotes the
     semi-flat K\"{a}hler-Einstein metric satisfying (\ref{e3.5}),   then, by Lemma \ref{le2}, $ P_{t}^{*}q_{\sigma}^{*} \omega_{\gamma,t}^{sf} $ is uniformly equivalent to $\tilde{\omega}^{o}$ on $ B_{I} \times \sqrt{-1}\mathbb{R}^{s}\times (U_{\gamma}\cap X_{0,I})$.  For any $U_{\gamma}$, since there are finite $U_{1},  \cdots,  U_{\varsigma}\in \{U_{\gamma}\}$ such that $(U\cup U_{1}\cup \cdots \cup U_{\varsigma}\cup U_{\gamma})\cap X_{t}$ is connected,  we have a  point $p_{t_{k},\gamma}\in X_{t_{k}}\cap U_{\gamma}$  such that ${\rm dist}_{\omega_{t_{k}}}(p_{t_{k},\gamma}, p_{t_{k}})\leq C_{\gamma}$ for a constant independent of $t_{k}$ by Proposition \ref{prop22}.    Thus there is a local isometric embedding $\iota_{\gamma}:(W_{U_{\gamma}}, g_{U_{\gamma}})\hookrightarrow (W, d_{W})$. Note that the restriction of  $g_{U_{\gamma}}$ on any $B_{I}  \times \{q\} $   is complete,  $g_{U_{\gamma}}|_{U_{\gamma}\cap X_{0,I}}$ is uniformly equivalent to $\tilde{\omega}^{o}|_{U_{\gamma}\cap X_{0,I}}=\tilde{\omega}_{0,I}|_{U_{\gamma}\cap X_{0,I}}$, and $\tilde{\omega}_{0,I}$ is complete on $X_{0,I}^{o}$ by (\ref{e2.5++}).   Therefore,  $\bigcup\limits_{\gamma} \iota_{\gamma}(W_{U_{\gamma}}, g_{U_{\gamma}})\subset (W, d_{W})$ is a complete Riemannian manifold, which implies that  $\bigcup\limits_{\gamma} \iota_{\gamma}(W_{U_{\gamma}}, g_{U_{\gamma}})= (W, d_{W})$.

Now we assume  $\dim_{\mathbb{C}}X_{0,I}=0$, i.e. $s=n$.
  Then $W_{U}=B_{I}$,  $\varphi_{0}=\varphi_{0}(x_{1}, \cdots, x_{n})$ is a function on $B_{I}$, and we denote $g_{B_{I}}=g_{W_{U}}$. We need the following lemma to finish the proof.

 \begin{lemma}\label{le3} If    $$\phi=\frac{\varphi_{0}}{2}-  \sum_{j=1}^{n}\log x_{j}-  \log(1-\sum_{j=1}^{n}x_{j}),  $$ then  $$g_{B_{I}}=\sum_{ij=1}^{n}\frac{\partial^{2} \phi}{\partial x_{i}\partial x_{j}}dx_{i}dx_{j}  $$ on $B_{I}$, and  $\phi$ is the unique solution of the  real Monge-Amp\`{e}re equation  $$ \det (\frac{\partial^{2} \phi}{\partial x_{i}\partial x_{j}})=\kappa e^{2 \phi},  \  \  \phi|_{\partial\overline{B}_{I}}=+\infty,  $$ for a constant $\kappa >0$, i.e.  $\phi$ is obtained  in Theorem \ref{CY}.
  \end{lemma}

 \begin{proof} Note that  $\frac{\partial \varphi_{0}}{\partial w_{j}}=\frac{\partial \varphi_{0}}{2\partial x_{j}}$, and $$  \frac{\partial^{2} \phi}{\partial x_{i}\partial x_{j}}=\frac{\partial^{2} \varphi_{0}}{2\partial x_{i}\partial x_{j}}+\frac{\delta_{ij}}{x_{j}^{2}}+\frac{1}{ (1-\sum\limits_{j=1}^{n}x_{j})^{2}}. $$
   By (\ref{e3.50}),  $g_{B_{I},ij}=\frac{\partial^{2} \phi}{\partial x_{i}\partial x_{j}}$.
   By Lemma \ref{le1} and Lemma \ref{le2},   we have $$\det (\frac{ \delta_{ij}}{x_{j}^{2}}+\frac{1}{(1-\sum\limits_{j=1}^{n}x_{j})^{2}} + \frac{\partial^{2} \varphi_{0}}{2\partial x_{i}\partial x_{j}})=\frac{e^{\varphi_{0}}2^{n} \eta'}{4^{n+1}(1-\sum\limits_{j=1}^{n}x_{j})^{2}\prod\limits_{j=1}^{n} x_{j}^{2}},  $$ where $\eta'>0$  is a constant.

 Now  Proposition 5.5 in \cite{CY} implies that $\varphi_{0}$ is the unique solution of (\ref{mae2}), which implies the uniqueness of $\phi$.
  \end{proof}

 Note that $g_{B_{I}}$ is a complete metric on $B_{I}$, and thus $(W, d_{W})$ is isometric to  $(B_{I},g_{B_{I}}) $.
By the uniqueness of $g_{B_{I}}$, we have the convergence of Proposition \ref{prop22} without passing to any   sequence $t_{k}$. We obtain the conclusion.
  \end{proof}

\section{Proof of Theorem    \ref{main2}}

 \begin{proof}[Proof of Theorem    \ref{main2}]
 Since $\mathcal{K}_{\mathcal{X}/\Delta}$   is ample,  there is  a section $\Omega_{t}$ of $\mathcal{K}_{\mathcal{X}/\Delta}$ such that $D \cap X_{0,I}=\emptyset$ where $D={\rm div}(\Omega_{t})$.
 Let  $U\subset \mathcal{X}$ be a neighborhood of $X_{0,I}$   isomorphic to an open subset of a toric variety  $ \mathcal{Y}_{\sigma}$, denoted still by $U$, such that $U\cap  X_{0,I'}$ is empty  for any $I'   \nsubseteq   I=\{1,\cdots, s+1\}$.  We assume that $D\cap U=\emptyset$ by shrinking $U$ if necessary.

  We adopt the construction  in Section 3.2. There is a toric variety $\mathcal{Y}_{\sigma}'\cong \mathbb{C}^{n+1}$ with coordinates $z_{0}, \cdots, z_{n}$, and a finite group $G=N/ N'$ acting on $\mathcal{Y}_{\sigma}'$.   Let $q_{\sigma}: \mathcal{Y}_{\sigma}'\rightarrow \mathcal{Y}_{\sigma}$ be the finite quotient by $G=N/ N'$, and $Y_{\sigma,t}\subset \mathcal{Y}_{\sigma}$ such that $Y_{\sigma, t}\cap U=X_{t}\cap U $ as in Section 3.2. Recall that  $q_{\sigma}^{-1}(Y_{\sigma,t})$ is given by $z_{0}\cdot\cdots\cdot z_{n}=t$, and $q_{\sigma}^{-1}( Y_{\sigma,t}\cap U)\subset B_{I}\times\sqrt{-1}(\mathbb{R}^{n}/\frac{2 \pi \mathbb{Z}^{n}}{\log |t|} )\subset q_{\sigma}^{-1}( Y_{\sigma,t})$, where $B_{I}$ is given by (\ref{e3.01}).

 For a $p=(p_{1}, \cdots, p_{n})\in B_{I}$, we define  an embedding $$\mathfrak{i}_{t}: B_{I}\times \sqrt{-1}(\mathbb{R}^{n}/\frac{2\pi \mathbb{Z}^{n}}{\log |t|} )\hookrightarrow \mathbb{C}^{n}/(2\pi\sqrt{-1}\mathbb{Z}^{n})=Y_{\infty}$$ by
     letting $\tilde{w}_{j}=(\log |t|)(w_{j}-p_{j})$, $j=1, \cdots, n$. We identify $B_{I}\times \sqrt{-1}(\mathbb{R}^{n}/\frac{2\pi \mathbb{Z}^{n}}{\log |t|} )$ with  the image $\mathfrak{i}_{t}(B_{I}\times \sqrt{-1}(\mathbb{R}^{n}/\frac{2\pi \mathbb{Z}^{n}}{\log |t|} ))\subset Y_{\infty}$ by $\mathfrak{i}_{t}$ without any confusion.

Assume that  $\tilde{\lambda}_{t}=\tilde{\lambda}_{t}(w_{1},\cdots,w_{n})$ is a family of functions convergence smoothly  to $\tilde{\lambda}_{0}=\tilde{\lambda}_{0}(x_{1},\cdots,x_{n})$ under the coordinates $w_{1},\cdots,w_{n}$ when $t\rightarrow 0$, i.e. $\frac{\partial^{k} \tilde{\lambda}_{t}}{\partial w_{j_{1}}^{k_{1}}\cdots \partial w_{j_{m}}^{k_{m}}}\rightarrow \frac{\partial^{k} \tilde{\lambda}_{0}}{\partial w_{j_{1}}^{k_{1}}\cdots \partial w_{j_{m}}^{k_{m}}}=\frac{\partial^{k} \tilde{\lambda}_{0}}{2^{k}\partial x_{j_{1}}^{k_{1}}\cdots \partial x_{j_{m}}^{k_{m}}}$.  Since  $w_{j}=p_{j}+(\log|t|)^{-1}\tilde{w}_{j}$ and $\frac{\partial \tilde{\lambda}_{t}}{\partial \tilde{w}_{j}}=(\log|t|)^{-1}\frac{\partial \tilde{\lambda}_{t}}{\partial w_{j}}$, we have $\tilde{\lambda}_{t} \rightarrow \tilde{\lambda}_{0} (p_{1}, \cdots, p_{n})$ in the $C^{\infty}$-sense on any
         on any compact subset  $K'\subset \mathfrak{i}_{t}( B_{I}\times \sqrt{-1}(\mathbb{R}^{n}/\frac{2\pi \mathbb{Z}^{n}}{\log |t|} )) \subset Y_{\infty}$, when $t\rightarrow 0$.

 Since $\mathcal{Y}_{\sigma}$ has only Gorenstein orbifold singularities, for the local  generator $\Omega_{\sigma}\in \mathcal{O}(\mathcal{K}_{\mathcal{Y}_{\sigma}})$, $q_{\sigma}^{*}\Omega_{\sigma}$ is a $G$-invariant  no-where vanishing holomorphic $(n+1,0)$-form on $\mathcal{Y}'_{\sigma}$, and thus $$q_{\sigma}^{*}\Omega_{t}=\Omega_{\sigma} \otimes (dt)^{-1}=  \zeta  \frac{dz_{1}\wedge \cdots \wedge  dz_{n}}{z_{1}\cdot\cdots\cdot z_{n} },$$ on $q_{\sigma}^{-1}(X_{t} \cap U)$,  where $\zeta>0$ is a holomorphic function on $\mathcal{Y}'_{\sigma}$.
   Note that $\zeta (w_{1}, \cdots, w_{n}) \rightarrow \zeta(0)$ in the $C^{\infty}$-sense by the argument in the proof of Lemma \ref{le1}. Thus
    $$q_{\sigma}^{*}\Omega_{t}=\zeta   d\tilde{w}_{1}\wedge \cdots \wedge  d\tilde{w}_{n} \rightarrow \Omega_{\infty}= \zeta (0)  d\tilde{w}_{1}\wedge \cdots \wedge  d\tilde{w}_{n}, $$ in the  $C^{\infty}$-sense, when $t\rightarrow 0$.

            If we denote $L_{0}=\{0\}\times \sqrt{-1}(\mathbb{R}^{n}/(2\pi \mathbb{Z}^{n} ) )$, then for any $|t|\ll 1$, there is a $\vartheta_{t}\in \mathbb{R}$ such that $e^{\sqrt{-1}\vartheta_{t}}\int_{L_{0}}q_{\sigma}^{*}\Omega_{t}\in\mathbb{ R}$, which implies that $\int_{L_{0}} {\rm Im}( e^{\sqrt{-1}\vartheta_{0}}\Omega_{\infty})=0$, $\int_{L_{0}} {\rm Im}( e^{\sqrt{-1}\vartheta_{t}}q_{\sigma}^{*}\Omega_{t})=0$,   and   $[ {\rm Im}(e^{\sqrt{-1}\vartheta_{t}}q_{\sigma}^{*}\Omega_{t})|_{L_{0}}]=0$ in $H^{n}(L_{0}, \mathbb{R})$.
            Since $e^{\sqrt{-1}\vartheta_{0}}\zeta(0)$ is a constant, we have   $$ {\rm Im}( e^{\sqrt{-1}\vartheta_{0}}\Omega_{\infty})|_{L_{0}}={\rm Im}( e^{\sqrt{-1}\vartheta_{0}} \zeta (0)  d\tilde{w}_{1}\wedge \cdots \wedge  d\tilde{w}_{n})|_{L_{0}} \equiv 0.$$

   By Lemma \ref{le1} and Lemma \ref{le2}, $$(\log |t|)^{2}q_{\sigma}^{*} \omega_{t}\rightarrow \frac{\sqrt{-1}}{2} \sum_{ij=1}^{n} ( \frac{ \delta_{ij}}{p_{j}^{2}}+ \frac{1}{(1-\sum\limits_{j=1}^{s}p_{j})^{2} }+\frac{\partial^{2} \varphi_{0}}{2 \partial x_{i} \partial x_{j}}(p) )   d\tilde{w}_{i}\wedge  d\bar{\tilde{w}}_{j}= \omega_{\infty},$$ in the $C^{\infty}$-sense on any compact subset $K'$  on $Y_{\infty}$, when $t\rightarrow 0$. Since the curvature of $ \omega_{t}$ are uniformly bounded independent of $t$, we have that
        $\omega_{\infty}$ is a flat   metric  on  $ Y_{\infty}$.
    A direct calculation shows  $\omega_{\infty}|_{L_{0}}\equiv 0$.
     Note that for any $A\in H_{2}(L_{0}, \mathbb{Z})$,
     $$ |(\log |t|)^{2}\int_{A}q_{\sigma}^{*}\omega_{t}|\rightarrow |\int_{A}\omega_{\infty}|=0,  \ \ {\rm and} $$ $$  \int_{A}q_{\sigma}^{*}\omega_{t}=2\pi  \int_{q_{\sigma}(A)}c_{1}(\mathcal{K}_{X_{t}}) \in 2 \pi\mathbb{Z}. $$ Thus $\int_{A}q_{\sigma}^{*}\omega_{t}=0$,  and  $[ q_{\sigma}^{*} \omega_{t}|_{L_{0}}]=0$ in $H^{2}(L_{0}, \mathbb{R})$.
  By Theorem 10.8 in \cite{GHJ}, we obtain  a family of generalized special lagrangian submanifolds $\tilde{L}_{t} \subset  B_{I}\times \sqrt{-1}(\mathbb{R}^{n}/\frac{2\pi \mathbb{Z}^{n}}{\log |t|} )$ respecting to $q_{\sigma}^{*}\omega_{t}$ and $e^{\sqrt{-1}\vartheta_{t}}q_{\sigma}^{*}\Omega_{t}$ for $|t|\ll 1$. We obtain the conclusion by letting  $L_{t}= q_{\sigma}(\tilde{L}_{t})$.
  \end{proof}

\end{document}